\documentclass[preprint,12pt]{elsarticle}

 \usepackage{amsmath,amsthm,amssymb}
 \usepackage{makeidx,epsfig,lscape}
 \usepackage{color,colortbl}
 \usepackage{fancyhdr}
 \usepackage{xcolor,pict2e}
 \usepackage{ragged2e}
\usepackage{times}

\usepackage{newtxtext,newtxmath}

 \definecolor{myaqua}{rgb}{0.0,0.5,0.55}
 \definecolor{lightaqua}{rgb}{0.75,0.95,0.95}

 \usepackage[colorlinks = true,
            linkcolor = myaqua,
            urlcolor  = blue,
            citecolor = myaqua]{hyperref}

\usepackage{caption}
\usepackage{floatrow}
\newtheorem{theorem}{Theorem}

\newtheorem{lemma}{Lemma}

\newtheorem{definition}{Definition}
\newtheorem{example}{Example}
\newtheorem{algorithm}{Algorithm}

\usepackage{graphicx}
\usepackage{tikz}
\usetikzlibrary{shapes.geometric,arrows}
\def\bt{\begin{tabular}}
\def\et{\end{tabular}}
\def\and{\mbox{ and }}

\def\1{{\bf 1}}

 \def\sectionn#1{\refstepcounter{section}{

 \vskip 6mm

 \noindent\Large\bf\thesection. #1}

 \vskip 3mm}

\usepackage[utf8]{inputenc}
\begin{document}
\begin{frontmatter}

\title{Collatz Theorem}
\author{Dagnachew Jenber Negash}
\address{Addis Ababa Science and Technology University\\Addis Ababa, Ethiopia\\Email: djdm$\_$101979@yahoo.com}








\begin{abstract}
\footnotesize{This paper studies the proof of Collatz conjecture for some set of sequence of odd numbers with infinite number of elements. These set generalized to the set which contains all positive odd integers. This extension assumed to be the proof of the full conjecture, using the concept of mathematical induction. In section 9, the Collatz conjecture is proved again using mathematical induction. Therefore Collatz conjecture is proved two times. Finally two algorithms with GNU Octave code are given to check the validity of the proof for some particular numbers, code of Algorithm 1 is the restatement of Collatz conjecture function which is used to check the validity of our newly generated Collatz function for some particular terms which is provided by Algorithm 2. Therefore go from Algorithm 2 to Algorithm 1 for checking the sequence of numbers and the number of steps needed to reach to 1 according to the conjecture.} 
\end{abstract}
\begin{keyword}
Collatz conjecture; Recurrence relations; Concept of mathematical induction.\\
MSC2010: 65Q30\sep 97N30\sep 39A12\sep 11B75
\end{keyword}
\end{frontmatter}


\sectionn{Introduction}
\label{S:1}
\justify
The Collatz conjecture is an unsolved conjecture in
mathematics. It is named after Lothar Collatz, who first
proposed it in 1937. The conjecture is also known as the $3n +
1$ conjecture, the Ulam conjecture (after Stanislaw Ulam), the
Syracuse problem, as the hailstone sequence or hailstone
numbers, or as Wondrous numbers per Godel, Escher, Bach.
It asks whether a certain kind of number sequence always
ends in the same way, regardless of the starting number.
The problem is related to a wide range
of topics in mathematics, including number theory, computability theory, and the analysis of dynamic systems.
Paul Erdos said about the Collatz conjecture, "Mathematics is
not yet ready for such problems." He offered \$500 for its
solution. (Lagarias 1985)
\sectionn{Statement of the problem}
\label{sec:2}
{ \fontfamily{times}\selectfont
 \noindent
\justify 
Consider the following operation on an arbitrary positive integer\cite{Guye}\cite{Collatz}:
\begin{itemize}
\item If the number is even, divide it by two.
\item If the number is odd, triple it and add one.
\end{itemize}
\begin{equation*}
f(n)=\left\{\begin{array}{ll}
n/2 & \text{ if }n \text{ is even,}\\
3n+1& \text{ if }n \text{ is odd.}
\end{array}\right.
\end{equation*}
Form a sequence by performing this operation repeatedly, beginning with
any positive integer.
\begin{itemize}

\item Example: n = 6 produces the sequence
\end{itemize}
\begin{equation*}
6\text{, }3\text{, } 10\text{, } 5\text{, } 16\text{, } 8\text{, } 4\text{, } 2\text{, } 1\text{, } 4\text{, } 2\text{, } 1\text{, }\cdots
\end{equation*}
\begin{itemize}
\item The Collatz conjecture is:\\
This process will eventually reach the number 1, regardless of which
positive integer is chosen initially.
\end{itemize}

\sectionn{Some Examples}
\label{sec:3}

{ \fontfamily{times}\selectfont
 \noindent
\begin{itemize}
\item n = 11 produces the sequence
\end{itemize}
\begin{equation*}
11, 34, 17, 52, 26, 13, 40, 20, 10, 5, 16, 8, 4, 2, 1.
\end{equation*}
which means we operated Collatz function $14$ times to reach to the number $1$.
\begin{itemize}
\item n = 27 produces the sequence
\end{itemize}
\begin{equation*}
27,82,41,124,62,31,94,47,142,71,214,107,322,161,484,242,121,
\end{equation*}
\begin{equation*}
364,182,91,274,137,412,206,103,310,155,466,233,700,350,175,
\end{equation*}
\begin{equation*}
526,263,790,395,1186,593,1780,890,445,1336,668,334,167,502,
\end{equation*}
\begin{equation*}
251,754,377,1132,566,283,850,425,1276,638,319,958,479,1438,
\end{equation*}
\begin{equation*}
719,2158,1079,3238,1619,4858,2429,7288,3644,1822,911,2734,
\end{equation*}
\begin{equation*}
1367,4102,2051,6154,3077,9232,4616,2308,1154,577,1732,866,
\end{equation*}
\begin{equation*}
433,1300,650,325,976,488,244,122,61,184,92,46,23,70,35,106,
\end{equation*}
\begin{equation*}
53,160,80,40,20,10,5,16,8,4,2,1.\hspace{45mm}
\end{equation*}that is, we operated or used Collatz function repeatdely $111$ times to reach to the number $1$.



 \newpage
\sectionn{Supporting arguments for the conjecture}
\label{sec:4}

{ \fontfamily{times}\selectfont
 \noindent
 
\begin{itemize}
\item Experimental evidence\cite{Garner}\cite{Lagarias}:
\end{itemize}
\justify
The conjecture has been checked by computer for all starting values
up to $19\times 2^{58} \approx 5.48\times10^{18}$
\begin{itemize}
\item A probabilistic argument:
\end{itemize}
\justify
One can show that each odd number in a sequence is on average
$3/4$ of the previous one, so every sequence should decrease in the
long run\cite{Lagariass}.\justify
This is not a proof because Collatz sequences are not produced by
random events.
\sectionn{One Definition}
\label{sec:5}
Through out this paper we will use the definition
\begin{definition}
Define $N_k$ as the number of collatz function operation for the number $k$ to get $1$.
\end{definition}
\sectionn{Proof of collatz analysis for some sequence of numbers }
\label{sec:6}

{ \fontfamily{times}\selectfont
 \noindent} 
 \justify
 We know that Collatz analysis is true for all numbers in the set $\{a_n: a_n=2^{2n}\text{ or }a_n=2^{2n-1},n\in \mathbb{N} \}$. Here I want to show that Collatz analysis is true for all numbers in the set $\{a_n: a_n=(2k+1)2^{2n}\text{ or }a_n=(2k+1)2^{2n-1},n,k\in \mathbb{N} \}$, here we have to be sure that Collatz conjecture becomes collatz Theorem if it is proved for this set. Before the conclusion let's see some sort of sequence of numbers and Collatz analysis is definitely true for these set of sequence of numbers. Let's get started from the following figure and then we will discuss, give some analysis based on our objective on the figure and form one sequence function which contains all numbers in the sequence that shows Collatz analysis is true. The following theorem shows the truth of collatz conjecture on some recurrence relations. The recurrence relations are inter-related, which means one recurrence relation formed from the previous one. Therefore if we proved that one recurrence relation satisfies Collatz conjecture then so the next, and then so on. Hence from the concept of mathematical induction, we will prove Collatz conjecture for all odd natural numbers, then the proof for all even natural number is straight forward, because all even number have the form $(2k+1)2^{n}$. To do this let's get started from $\bf{Theorem \text{ }1.}$
\begin{theorem}
Collatz conjecture is true for all odd numbers in the set of the recurrence relation 
\begin{equation*}
 \{ a_n: a_n-a_{n-1}=2^{2n},a_0=1\text{ and }n\in \mathbb{N}\}
\end{equation*}

\begin{equation*}
\Rightarrow a_n=\frac{1}{3}\bigg[4^{n+1}-1\bigg],\text{ }n\geq 0
\end{equation*} and for each $a_{n}$, there are $2n+3$ number of steps (or Collatz function operation) needed to get $1$.   
\end{theorem}
\begin{proof}   
\newpage
\tikzstyle{rect}=[draw,rectangle,fill=blue!20,text width=3em,text centered, minimum height=2em]
\tikzstyle{ellip}=[draw,ellipse,fill=white!20, minimum height=2em]
\tikzstyle{circ}=[draw, circle, fill=white!20, minimum width=2pt,inner sep=3pt]
\tikzstyle{diam}=[draw,diamond,fill=white!20,text width=8em, text badly centered, inner sep=0pt]
\tikzstyle{line}=[draw,-latex']
\begin{figure}[h!]
\begin{center}
\begin{tikzpicture}[node distance=1cm, auto]
\node [rect, rounded corners] (step1) {$1$};
\node [rect, rounded corners, above of=step1,node distance=2cm] (step2) {$1(2)$};
\node [rect, rounded corners, above of=step2,node distance=2cm] (step3) {$1(2^2)$};
\node [rect, rounded corners, above of=step3,node distance=2cm] (step4) {$1(2^3)$};
\node [rect, rounded corners, above of=step4,node distance=2cm] (step5) {$1(2^4)$};
\node [rect, rounded corners, above of=step5,node distance=2cm] (step6) {$1(2^5)$};
\node [rect, rounded corners, above of=step6,node distance=2cm] (step7) {$1(2^6)$};
\node [rect, rounded corners, above of=step7,node distance=2cm] (step8) {$1(2^7)$};
\node [rect, rounded corners, above of=step8,node distance=2cm] (step9) {$1(2^8)$};
\node [rect, rounded corners, above of=step9,node distance=2cm] (step10) {$1(2^9)$};
\node [rect, right of=step3,node distance=2cm,color=red] (step11) {$\color{white}1$};
\node [rect, right of=step11,node distance=2cm] (step12) {$1(2)$};
\node [rect, right of=step12,node distance=2cm] (step13) {$1(2^2)$};
\node [rect, right of=step13,node distance=2cm] (step14) {$1(2^3)$};
\node [rect, right of=step14,node distance=2cm] (step15) {$1(2^4)$};
\node [rect, right of=step15,node distance=2cm] (step16) {$1(2^5)$};
\node [rect, right of=step16,node distance=2cm] (step17) {$1(2^6)$};
\node [rect, right of=step17,node distance=2cm] (step18) {$1(2^7)$};
\node [rect, right of=step5,node distance=2cm,color=red] (step19) {$\color{white}5$};
\node [rect, right of=step19,node distance=2cm] (step20) {$5(2)$};
\node [rect, right of=step20,node distance=2cm] (step21) {$5(2^2)$};
\node [rect, right of=step21,node distance=2cm] (step22) {$5(2^3)$};
\node [rect, right of=step22,node distance=2cm] (step23) {$5(2^4)$};
\node [rect, right of=step23,node distance=2cm] (step24) {$5(2^5)$};
\node [rect, right of=step24,node distance=2cm] (step25) {$5(2^6)$};
\node [rect, right of=step25,node distance=2cm] (step26) {$5(2^7)$};

\node [rect, right of=step7,node distance=2cm,color=red] (step27) {$\color{white}21$};
\node [rect, right of=step27,node distance=2cm] (step28) {$21(2)$};
\node [rect, right of=step28,node distance=2cm] (step29) {$21(2^2)$};
\node [rect, right of=step29,node distance=2cm] (step30) {$21(2^3)$};
\node [rect, right of=step30,node distance=2cm] (step31) {$21(2^4)$};
\node [rect, right of=step31,node distance=2cm] (step32) {$21(2^5)$};
\node [rect, right of=step32,node distance=2cm] (step33) {$21(2^6)$};
\node [rect, right of=step33,node distance=2cm] (step34) {$21(2^7)$};

\node [rect, right of=step9,node distance=2cm,color=red] (step35) {\color{white} $85$};
\node [rect, right of=step35,node distance=2cm] (step36) {\tiny $85(2)$};
\node [rect, right of=step36,node distance=2cm] (step37) {\tiny $85(2^2)$};
\node [rect, right of=step37,node distance=2cm] (step38) {\tiny $85(2^3)$};
\node [rect, right of=step38,node distance=2cm] (step39) {\tiny $85(2^4)$};
\node [rect, right of=step39,node distance=2cm] (step40) {\tiny $85(2^5)$};
\node [rect, right of=step40,node distance=2cm] (step41) {\tiny $85(2^6)$};
\node [rect, right of=step41,node distance=2cm] (step42) {\tiny $85(2^7)$};

\path [line] (step1)--(step2);
\path [line] (step2)--(step3);
\path [line] (step3)--(step4);
\path [line] (step4)--(step5);
\path [line] (step5)--(step6);
\path [line] (step6)--(step7);
\path [line] (step7)--(step8);
\path [line] (step8)--(step9);
\path [line] (step9)--(step10);

\path [line] (step18)--(step17);
\path [line] (step17)--(step16);
\path [line] (step16)--(step15);
\path [line] (step15)--(step14);
\path [line] (step14)--(step13);
\path [line] (step13)--(step12);
\path [line] (step12)--(step11);
\path [line] (step11)--(step3);

\path [line] (step26)--(step25);
\path [line] (step25)--(step24);
\path [line] (step24)--(step23);
\path [line] (step23)--(step22);
\path [line] (step22)--(step21);
\path [line] (step21)--(step20);
\path [line] (step20)--(step19);
\path [line] (step19)--(step5);

\path [line] (step34)--(step33);
\path [line] (step33)--(step32);
\path [line] (step32)--(step31);
\path [line] (step31)--(step30);
\path [line] (step30)--(step29);
\path [line] (step29)--(step28);
\path [line] (step28)--(step27);
\path [line] (step27)--(step7);

\path [line] (step42)--(step41);
\path [line] (step41)--(step40);
\path [line] (step40)--(step39);
\path [line] (step39)--(step38);
\path [line] (step38)--(step37);
\path [line] (step37)--(step36);
\path [line] (step36)--(step35);
\path [line] (step35)--(step9);

\end{tikzpicture}
\caption{\small Proof of Collatz conjecture for the sequence of numbers $\{1,5,21,85,\cdots \}$}
\end{center}

\end{figure}
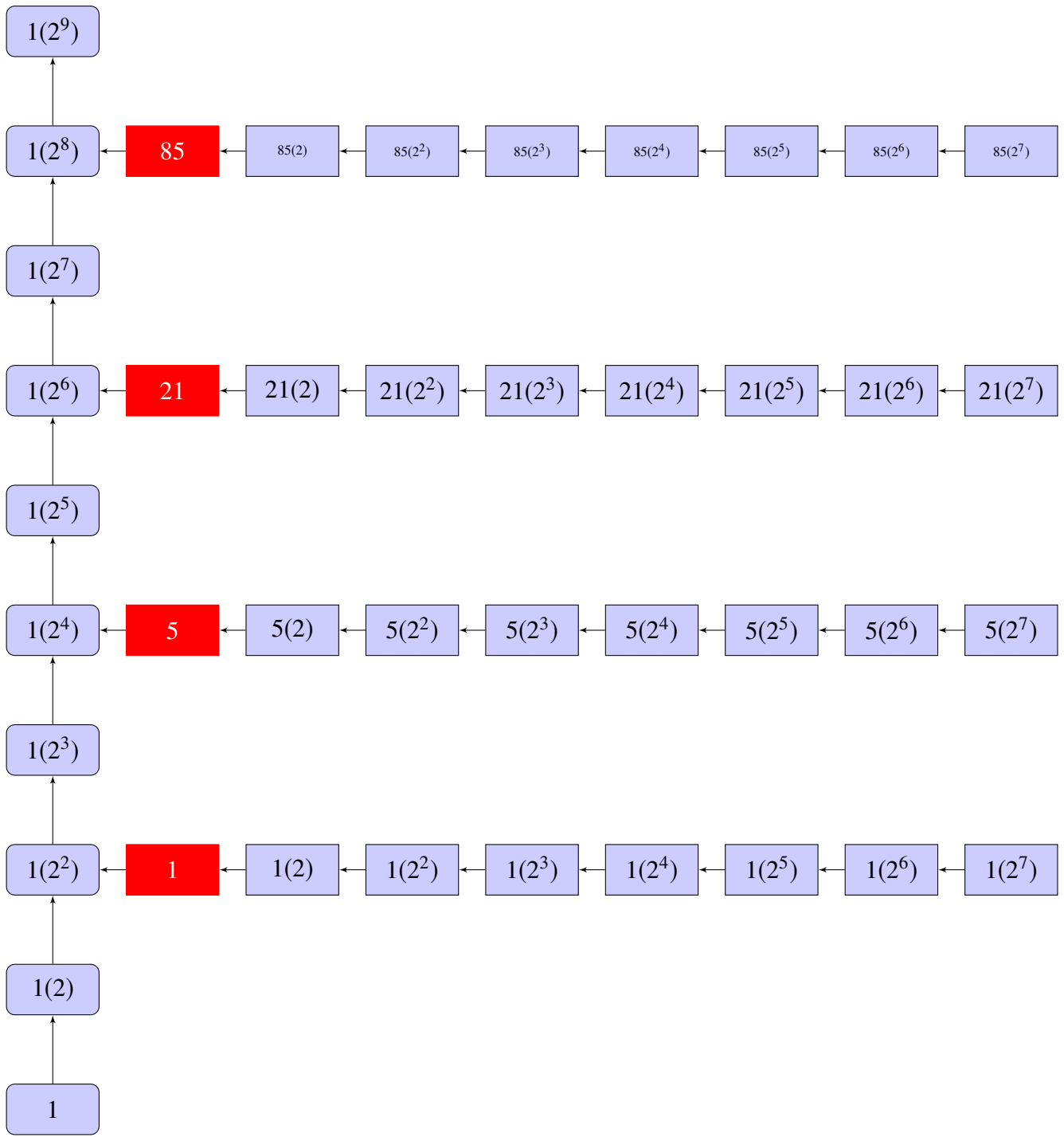
\justify
As you see from the figure we have sequence of numbers $\{1,5,21,85,\cdots \}$. All elements of this sequence are odd numbers with $a_0=1,a_1=5,a_2=21,a_3=85,\cdots$, and $3(1)+1=2^2,3(5)+1=2^{4},3(21)+1=2^6,3(85)+1=2^8,\cdots$. The question is how this sequence of numbers formed?. The answer is simple. You can observe that $5=1+2^2,21=5+2^4,85=21+2^6,\cdots$. Therefore the recurrence relation is $a_n=a_{n-1}+2^{2n}$ or $a_n-a_{n-1}=2^{2n}$ with $a_0=1$ and $n\in \mathbb{N}$.\\ Hence we can conclude that collatz analysis is true for all numbers in the set
\begin{equation*}
 \{ a_n: a_n-a_{n-1}=2^{2n},a_0=1\text{ and }n\in \mathbb{N}\}
\end{equation*}
If we solve this non-homogeneous recurrence relation we get,
\begin{equation*}
a_n=\frac{1}{3}\bigg[4^{n+1}-1\bigg],\text{ }n\geq 0
\end{equation*} 
This is the general term for the above sequence of numbers. We might ask "How many number of steps we need to reach to 1 for each number in the above sequence by applying Collatz function?". Here is the answer. For example for $n=0$, we have $a_0=1$, therefore according to Collatz conjecture $3(a_0)+1=2^2$ since $a_0$ is odd. As you can see from the above figure that we need to have $3$ steps, which means $2+1$ steps. Similarly for $a_1=5$ we need to have $5$ steps to reach to $1$ since $3(5)+1=2^{4}$, that is take the exponent of $2$ on the right hand side of $3(5)+1=2^4$ and add $1$, that is, $4+1$ steps. Thus for $a_2=21$ we need to have $7$ steps since $3(21)+1=2^{6}$ and $a_3=85$ needs to have $9$ steps since $3(85)+1=2^8$. You can check that this is true for the rest of numbers in the sequence by choosing randomly.
\end{proof}

Next what I'm going to do is track another sequence function. The question is how to get this sequence function?. Well, the answer is simple if I draw another figure just like the previous but I have to use the previous figure to draw the next. First of all, for $k \in \mathbb{N}$, let me find the least $\beta_1=a_k2^{2n} \text{ or }=a_k2^{2n-1}\text{ for }n=1$ such that $\beta_1-1$ should be divisible by $3$, where $a_k \in \{ a_n: a_n-a_{n-1}=2^{2n},a_0=1\text{ and }n\in \mathbb{N}\}$ and $(\beta_1-1)/3$ will be the initial term for the general term of the next sequence of numbers. So after two or three simple trials you can get that $a_k=5$ and $\beta_1=5(2)=10$ which is the least among others and you can verify that $\beta_1-1=10-1=9$ which is divisible by $3$ and $(\beta_1-1)/3=3=a_0$ is the initial term for the next sequence function. Now let's see the following figure to be more clear for the next and the previous work.
\begin{theorem}
Collatz conjecture is true for all odd numbers in the set of the recurrence relation 
\begin{equation*}
 \{ b_n: b_n-b_{n-1}=5(2^{2n-1}),b_0=3\text{ and }n\in \mathbb{N}\}
\end{equation*}
\begin{equation*}
\Rightarrow b_n=\frac{1}{3}\bigg[5(2^{2n+1})-1\bigg],\text{ }n\geq 0
\end{equation*}  and for each $b_{n}$, there are $N_{5}+2n+2$ number of steps (or Collatz function operation) needed to get $1$.   
\end{theorem} 
\begin{proof}
\newpage
\tikzstyle{rect}=[draw,rectangle,fill=blue!20,text width=3em,text centered, minimum height=2em]
\tikzstyle{ellip}=[draw,ellipse,fill=white!20, minimum height=2em]
\tikzstyle{circ}=[draw, circle, fill=white!20, minimum width=2pt,inner sep=3pt]
\tikzstyle{diam}=[draw,diamond,fill=white!20,text width=8em, text badly centered, inner sep=0pt]
\tikzstyle{line}=[draw,-latex']
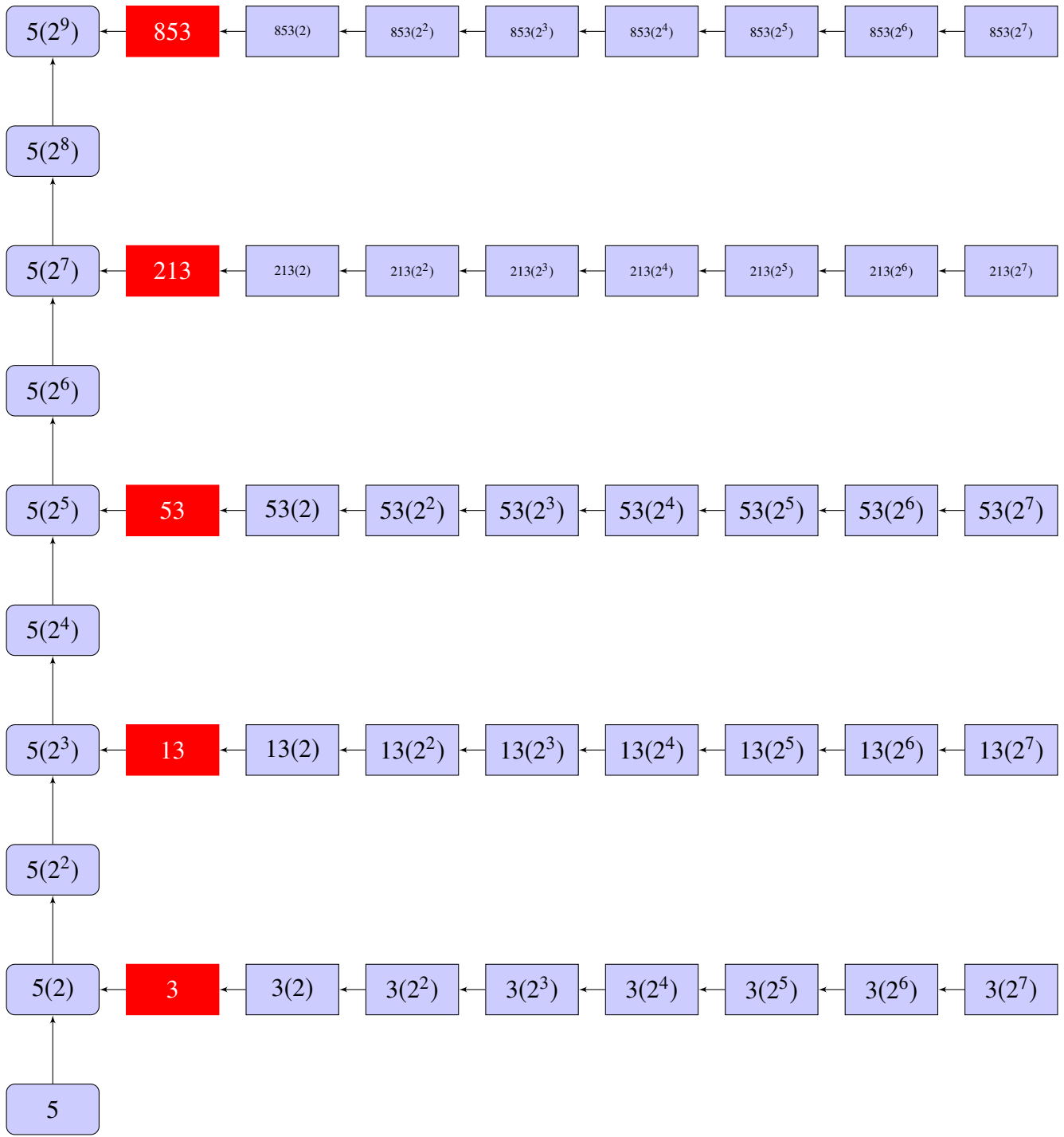
\begin{figure}[h!]
\begin{center}
\begin{tikzpicture}[node distance=1cm, auto]
\node [rect, rounded corners] (step1) {$5$};
\node [rect, rounded corners, above of=step1,node distance=2cm] (step2) {$5(2)$};
\node [rect, rounded corners, above of=step2,node distance=2cm] (step3) {$5(2^2)$};
\node [rect, rounded corners, above of=step3,node distance=2cm] (step4) {$5(2^3)$};
\node [rect, rounded corners, above of=step4,node distance=2cm] (step5) {$5(2^4)$};
\node [rect, rounded corners, above of=step5,node distance=2cm] (step6) {$5(2^5)$};
\node [rect, rounded corners, above of=step6,node distance=2cm] (step7) {$5(2^6)$};
\node [rect, rounded corners, above of=step7,node distance=2cm] (step8) {$5(2^7)$};
\node [rect, rounded corners, above of=step8,node distance=2cm] (step9) {$5(2^8)$};
\node [rect, rounded corners, above of=step9,node distance=2cm] (step10) {$5(2^9)$};
\node [rect, right of=step2,node distance=2cm,color=red] (step11) {$\color{white}3$};
\node [rect, right of=step11,node distance=2cm] (step12) {$3(2)$};
\node [rect, right of=step12,node distance=2cm] (step13) {$3(2^2)$};
\node [rect, right of=step13,node distance=2cm] (step14) {$3(2^3)$};
\node [rect, right of=step14,node distance=2cm] (step15) {$3(2^4)$};
\node [rect, right of=step15,node distance=2cm] (step16) {$3(2^5)$};
\node [rect, right of=step16,node distance=2cm] (step17) {$3(2^6)$};
\node [rect, right of=step17,node distance=2cm] (step18) {$3(2^7)$};
\node [rect, right of=step4,node distance=2cm,color=red] (step19) {$\color{white}13$};
\node [rect, right of=step19,node distance=2cm] (step20) {$13(2)$};
\node [rect, right of=step20,node distance=2cm] (step21) {$13(2^2)$};
\node [rect, right of=step21,node distance=2cm] (step22) {$13(2^3)$};
\node [rect, right of=step22,node distance=2cm] (step23) {$13(2^4)$};
\node [rect, right of=step23,node distance=2cm] (step24) {$13(2^5)$};
\node [rect, right of=step24,node distance=2cm] (step25) {$13(2^6)$};
\node [rect, right of=step25,node distance=2cm] (step26) {$13(2^7)$};

\node [rect, right of=step6,node distance=2cm,color=red] (step27) {$\color{white}53$};
\node [rect, right of=step27,node distance=2cm] (step28) {$53(2)$};
\node [rect, right of=step28,node distance=2cm] (step29) {$53(2^2)$};
\node [rect, right of=step29,node distance=2cm] (step30) {$53(2^3)$};
\node [rect, right of=step30,node distance=2cm] (step31) {$53(2^4)$};
\node [rect, right of=step31,node distance=2cm] (step32) {$53(2^5)$};
\node [rect, right of=step32,node distance=2cm] (step33) {$53(2^6)$};
\node [rect, right of=step33,node distance=2cm] (step34) {$53(2^7)$};

\node [rect, right of=step8,node distance=2cm,color=red] (step35) {\color{white} $213$};
\node [rect, right of=step35,node distance=2cm] (step36) {\tiny $213(2)$};
\node [rect, right of=step36,node distance=2cm] (step37) {\tiny $213(2^2)$};
\node [rect, right of=step37,node distance=2cm] (step38) {\tiny $213(2^3)$};
\node [rect, right of=step38,node distance=2cm] (step39) {\tiny $213(2^4)$};
\node [rect, right of=step39,node distance=2cm] (step40) {\tiny $213(2^5)$};
\node [rect, right of=step40,node distance=2cm] (step41) {\tiny $213(2^6)$};
\node [rect, right of=step41,node distance=2cm] (step42) {\tiny $213(2^7)$};

\node [rect, right of=step10,node distance=2cm,color=red] (step43) {\color{white} $853$};
\node [rect, right of=step43,node distance=2cm] (step44) {\tiny $853(2)$};
\node [rect, right of=step44,node distance=2cm] (step45) {\tiny $853(2^2)$};
\node [rect, right of=step45,node distance=2cm] (step46) {\tiny $853(2^3)$};
\node [rect, right of=step46,node distance=2cm] (step47) {\tiny $853(2^4)$};
\node [rect, right of=step47,node distance=2cm] (step48) {\tiny $853(2^5)$};
\node [rect, right of=step48,node distance=2cm] (step49) {\tiny $853(2^6)$};
\node [rect, right of=step49,node distance=2cm] (step50) {\tiny $853(2^7)$};

\path [line] (step1)--(step2);
\path [line] (step2)--(step3);
\path [line] (step3)--(step4);
\path [line] (step4)--(step5);
\path [line] (step5)--(step6);
\path [line] (step6)--(step7);
\path [line] (step7)--(step8);
\path [line] (step8)--(step9);
\path [line] (step9)--(step10);

\path [line] (step18)--(step17);
\path [line] (step17)--(step16);
\path [line] (step16)--(step15);
\path [line] (step15)--(step14);
\path [line] (step14)--(step13);
\path [line] (step13)--(step12);
\path [line] (step12)--(step11);
\path [line] (step11)--(step2);

\path [line] (step26)--(step25);
\path [line] (step25)--(step24);
\path [line] (step24)--(step23);
\path [line] (step23)--(step22);
\path [line] (step22)--(step21);
\path [line] (step21)--(step20);
\path [line] (step20)--(step19);
\path [line] (step19)--(step4);

\path [line] (step34)--(step33);
\path [line] (step33)--(step32);
\path [line] (step32)--(step31);
\path [line] (step31)--(step30);
\path [line] (step30)--(step29);
\path [line] (step29)--(step28);
\path [line] (step28)--(step27);
\path [line] (step27)--(step6);

\path [line] (step42)--(step41);
\path [line] (step41)--(step40);
\path [line] (step40)--(step39);
\path [line] (step39)--(step38);
\path [line] (step38)--(step37);
\path [line] (step37)--(step36);
\path [line] (step36)--(step35);
\path [line] (step35)--(step8);

\path [line] (step50)--(step49);
\path [line] (step49)--(step48);
\path [line] (step48)--(step47);
\path [line] (step47)--(step46);
\path [line] (step46)--(step45);
\path [line] (step45)--(step44);
\path [line] (step44)--(step43);
\path [line] (step43)--(step10);

\end{tikzpicture}
\caption{\small Proof of Collatz conjecture for the sequence of numbers $\{3,13,53,213,853\cdots \}$ combined with Fig 1.}
\end{center}
\end{figure}

\justify
As you can see from figure $2$, we have sequence of numbers $\{3,13,53,213,853\cdots \}$. All elements of this sequence are odd numbers with $b_0=3,b_1=13,b_2=53,b_3=213,b_4=853\cdots$, and $3(3)+1=5(2^1),3(13)+1=5(2^{3}),3(53)+1=5(2^5),3(213)+1=5(2^7),3(853)+1=5(2^9)\cdots$. The question is how this sequence of numbers formed?. The answer is simple. You can observe that $13=3+5(2^1),53=13+5(2^3),213=53+5(2^5),853=213+5(2^7)\cdots$. Therefore the recurrence relation is $b_n=a_{n-1}+5(2^{2n-1})$ or $b_n-b_{n-1}=5(2^{2n-1})$ with $b_0=3$ and $n\in \mathbb{N}$.\\ Hence we can conclude that Collatz analysis is true for all numbers in the set
\begin{equation*}
 \{ b_n: b_n-b_{n-1}=5(2^{2n-1}),b_0=3\text{ and }n\in \mathbb{N}\}
\end{equation*}
If we solve this non-homogeneous recurrence relation we get,
\begin{equation*}
b_n=\frac{1}{3}\bigg[5(2^{2n+1})-1\bigg],\text{ }n\geq 0
\end{equation*} 
This is the general term for the above sequence of numbers. We might ask "How many number of steps we need to reach to $1$ for each number in the above sequence by applying Collatz function?". Here is the answer. For example for $n=0$, we have $b_0=3$, therefore according to Collatz conjecture $3(b_0)+1=10=5(2)$ since $b_0$ is odd. As you can see from the above figure that we need to have $7$ steps, which means $2+5$ steps, where $2=1+1$ is the exponent of $2$ plus $1$ from the equation on the right side of $3(b_0)+1=10=5(2)$ and $5$ is the number steps needed for $5$, so $b_0=3$ needs to have $7=5+2$ number of steps to reach to $1$ by applying Collatz  function repeatedly. Similarly for $b_1=13$ we need to have $5+4=9$ steps to reach to $1$ since $3(13)+1=5(2^{3})$, that is take the exponent of $2$ on the right hand side of $3(13)+1=5(2^3)$ and add $1$, that is, $3+1$ steps plus $5$ steps. Thus for $b_2=53$ we need to have $11$ steps since $3(53)+1=5(2^{5})$ and $b_3=213$ needs to have $13$ steps since $3(213)+1=5(2^7)$. You can check that this is true for the rest of numbers in the sequence by choosing randomly.
\end{proof}
Next, what I'm going to do is, track another sequence function. The question is how to get this sequence function?. Well, the answer is simple if I draw another figure $3$ just like the previous but I have to use the previous figures to draw the next. First of all, for $k\in \mathbb{N}$, let me find the least $\beta_2=b_k2^{2n} \text{ or }=b_k2^{2n-1}\text{ for }n=1$ such that $\beta_2-1$ should be divisible by $3$, where $b_k \in \{ b_n: b_n-b_{n-1}=5(2^{2n-1}),b_0=3\text{ and }n\in \mathbb{N}\}$ and $(\beta_2-1)/3$ will be the initial term for the general term of the next sequence of numbers. So after two or three simple trials you can get that $b_k=13$ and $\beta_2=13(2^2)=52$ which is the least among others and you can verify that $\beta_2-1=52-1=51$ which is divisible by $3$ and $(\beta_2-1)/3=17=b_0$ is the initial term for the next sequence function. Now let's see the following figure $3$ to be more clear for the next and the previous work.
\begin{theorem}
Collatz conjecture is true for all odd numbers in the set of the recurrence relation 
\begin{equation*}
 \{ c_n: c_n-c_{n-1}=13(2^{2n}),c_0=17\text{ and }n\in \mathbb{N}\}
\end{equation*}
\begin{equation*}
c_n=\frac{1}{3}\bigg[13(2^{2n+2})-1\bigg],\text{ }n\geq 0
\end{equation*}   and for each $c_{n}$, there are $N_{13}+2n+3$ number of steps (or Collatz function operation) needed to get $1$.   
\end{theorem} 
 \begin{proof}
\newpage
\tikzstyle{rect}=[draw,rectangle,fill=blue!20,text width=3em,text centered, minimum height=2em]
\tikzstyle{ellip}=[draw,ellipse,fill=white!20, minimum height=2em]
\tikzstyle{circ}=[draw, circle, fill=white!20, minimum width=2pt,inner sep=3pt]
\tikzstyle{diam}=[draw,diamond,fill=white!20,text width=8em, text badly centered, inner sep=0pt]
\tikzstyle{line}=[draw,-latex']
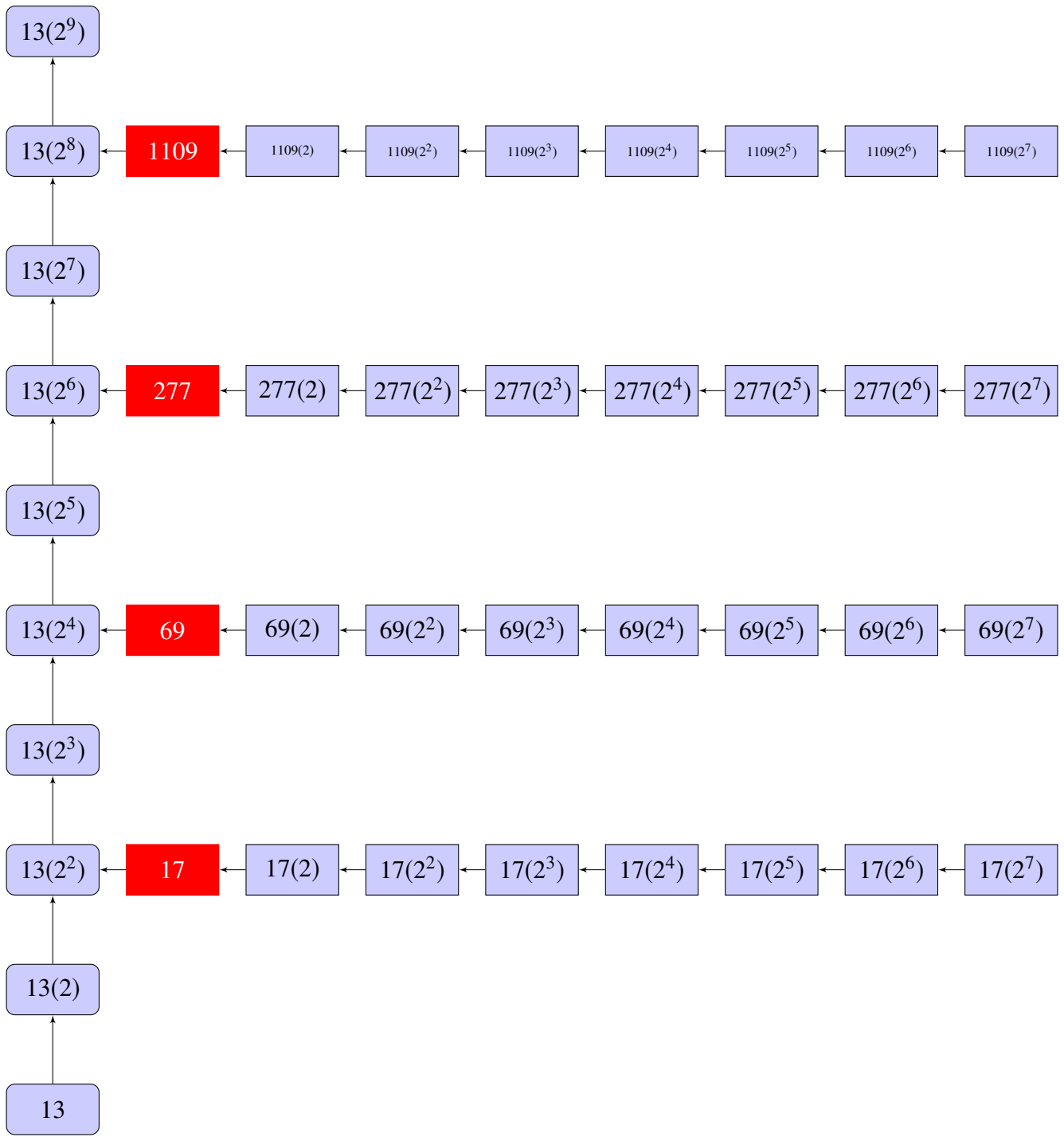
\begin{figure}[h!]
\begin{center}
\begin{tikzpicture}[node distance=1cm, auto]
\node [rect, rounded corners] (step1) {$13$};
\node [rect, rounded corners, above of=step1,node distance=2cm] (step2) {$13(2)$};
\node [rect, rounded corners, above of=step2,node distance=2cm] (step3) {$13(2^2)$};
\node [rect, rounded corners, above of=step3,node distance=2cm] (step4) {$13(2^3)$};
\node [rect, rounded corners, above of=step4,node distance=2cm] (step5) {$13(2^4)$};
\node [rect, rounded corners, above of=step5,node distance=2cm] (step6) {$13(2^5)$};
\node [rect, rounded corners, above of=step6,node distance=2cm] (step7) {$13(2^6)$};
\node [rect, rounded corners, above of=step7,node distance=2cm] (step8) {$13(2^7)$};
\node [rect, rounded corners, above of=step8,node distance=2cm] (step9) {$13(2^8)$};
\node [rect, rounded corners, above of=step9,node distance=2cm] (step10) {$13(2^9)$};
\node [rect, right of=step3,node distance=2cm,color=red] (step11) {\color{white}$17$};
\node [rect, right of=step11,node distance=2cm] (step12) {$17(2)$};
\node [rect, right of=step12,node distance=2cm] (step13) {$17(2^2)$};
\node [rect, right of=step13,node distance=2cm] (step14) {$17(2^3)$};
\node [rect, right of=step14,node distance=2cm] (step15) {$17(2^4)$};
\node [rect, right of=step15,node distance=2cm] (step16) {$17(2^5)$};
\node [rect, right of=step16,node distance=2cm] (step17) {$17(2^6)$};
\node [rect, right of=step17,node distance=2cm] (step18) {$17(2^7)$};
\node [rect, right of=step5,node distance=2cm,color=red] (step19) {\color{white}$69$};
\node [rect, right of=step19,node distance=2cm] (step20) {$69(2)$};
\node [rect, right of=step20,node distance=2cm] (step21) {$69(2^2)$};
\node [rect, right of=step21,node distance=2cm] (step22) {$69(2^3)$};
\node [rect, right of=step22,node distance=2cm] (step23) {$69(2^4)$};
\node [rect, right of=step23,node distance=2cm] (step24) {$69(2^5)$};
\node [rect, right of=step24,node distance=2cm] (step25) {$69(2^6)$};
\node [rect, right of=step25,node distance=2cm] (step26) {$69(2^7)$};

\node [rect, right of=step7,node distance=2cm,color=red] (step27) {\color{white}$277$};
\node [rect, right of=step27,node distance=2cm] (step28) {$277(2)$};
\node [rect, right of=step28,node distance=2cm] (step29) {$277(2^2)$};
\node [rect, right of=step29,node distance=2cm] (step30) {$277(2^3)$};
\node [rect, right of=step30,node distance=2cm] (step31) {$277(2^4)$};
\node [rect, right of=step31,node distance=2cm] (step32) {$277(2^5)$};
\node [rect, right of=step32,node distance=2cm] (step33) {$277(2^6)$};
\node [rect, right of=step33,node distance=2cm] (step34) {$277(2^7)$};

\node [rect, right of=step9,node distance=2cm,color=red] (step35) {\color{white} $1109$};
\node [rect, right of=step35,node distance=2cm] (step36) {\tiny $1109(2)$};
\node [rect, right of=step36,node distance=2cm] (step37) {\tiny $1109(2^2)$};
\node [rect, right of=step37,node distance=2cm] (step38) {\tiny $1109(2^3)$};
\node [rect, right of=step38,node distance=2cm] (step39) {\tiny $1109(2^4)$};
\node [rect, right of=step39,node distance=2cm] (step40) {\tiny $1109(2^5)$};
\node [rect, right of=step40,node distance=2cm] (step41) {\tiny $1109(2^6)$};
\node [rect, right of=step41,node distance=2cm] (step42) {\tiny $1109(2^7)$};

\path [line] (step1)--(step2);
\path [line] (step2)--(step3);
\path [line] (step3)--(step4);
\path [line] (step4)--(step5);
\path [line] (step5)--(step6);
\path [line] (step6)--(step7);
\path [line] (step7)--(step8);
\path [line] (step8)--(step9);
\path [line] (step9)--(step10);

\path [line] (step18)--(step17);
\path [line] (step17)--(step16);
\path [line] (step16)--(step15);
\path [line] (step15)--(step14);
\path [line] (step14)--(step13);
\path [line] (step13)--(step12);
\path [line] (step12)--(step11);
\path [line] (step11)--(step3);

\path [line] (step26)--(step25);
\path [line] (step25)--(step24);
\path [line] (step24)--(step23);
\path [line] (step23)--(step22);
\path [line] (step22)--(step21);
\path [line] (step21)--(step20);
\path [line] (step20)--(step19);
\path [line] (step19)--(step5);

\path [line] (step34)--(step33);
\path [line] (step33)--(step32);
\path [line] (step32)--(step31);
\path [line] (step31)--(step30);
\path [line] (step30)--(step29);
\path [line] (step29)--(step28);
\path [line] (step28)--(step27);
\path [line] (step27)--(step7);

\path [line] (step42)--(step41);
\path [line] (step41)--(step40);
\path [line] (step40)--(step39);
\path [line] (step39)--(step38);
\path [line] (step38)--(step37);
\path [line] (step37)--(step36);
\path [line] (step36)--(step35);
\path [line] (step35)--(step9);

\end{tikzpicture}
\caption{\small Proof of Collatz conjecture for the sequence of numbers $\{17,69,277,1109\cdots \}$ combined with Fig 1. and Fig 2.}
\end{center}
\end{figure}

\justify
As you can see from figure $3$ we have sequence of numbers $\{17,69,277,1109\cdots \}$. All elements of this sequence are odd numbers with $c_0=17,c_1=69,c_2=277,c_3=1109\cdots$, and $3(17)+1=13(2^2),3(69)+1=13(2^{4}),3(277)+1=13(2^6),3(1109)+1=13(2^8),\cdots$. The question is how this sequence of numbers formed?. The answer is simple. You can observe that $69=17+13(2^2),277=69+13(2^4),1109=277+13(2^6),\cdots$. Therefore the general term is $c_n=a_{n-1}+13(2^{2n})$ or $c_n-a_{n-1}=13(2^{2n})$ with $c_0=17$ and $n\in \mathbb{N}$.\\ Hence we can conclude that Collatz analysis is true for all numbers in the set
\begin{equation*}
 \{ c_n: c_n-c_{n-1}=13(2^{2n}),c_0=17\text{ and }n\in \mathbb{N}\}
\end{equation*}

If we solve this non-homogeneous recurrence relation we get,
\begin{equation*}
c_n=\frac{1}{3}\bigg[13(2^{2n+2})-1\bigg],\text{ }n\geq 0
\end{equation*} 
This is the general term for the above sequence of numbers. We might ask "How many number of steps we need to reach to 1 for each number in the above sequence by applying Collatz function?". Here is the answer. For example, for $n=0$, we have $c_0=17$, therefore according to Collatz conjecture $3(17)+1=52=13(2^2)$ since $c_0$ is odd. As you can see from the above figure that we need to have $12$ steps, which means $3+9$ steps, where $3=2+1$ is the exponent of $2$ plus $1$ from the equation on the right side of $3(c_0)+1=52=13(2^2)$ and $9$ is the number steps needed for $13$, so $c_0=17$ needs to have $12=9+3$ number of steps to reach to $1$ by applying Collatz  function repeatedly. Similarly for $c_1=69$ we need to have $9+5=14$ steps to reach to $1$ since $3(69)+1=13(2^{4})$, that is take the exponent of $2$ on the right hand side of $3(69)+1=13(2^4)$ and add $1$, that is, $4+1$ steps plus $9$ steps. Thus for $a_2=277$ we need to have $16$ steps since $3(277)+1=13(2^{6})$ and $a_3=1109$ needs to have $18$ steps since $3(1109)+1=13(2^8)$. You can check that this is true for the rest of numbers in the sequence by choosing randomly.
\end{proof}
Next, what I'm going to do is track another sequence function. The question is how to get this sequence function?. Well, the answer is simple if I draw another picture just like the previous but I have to use the previous pictures to draw the next. First of all, for $k \in \mathbb{N}$, let me find the least $\beta_3=c_k2^{2n} \text{ or }=c_k2^{2n-1}\text{ for }n=1$ such that $\beta_3-1$ should be divisible by $3$, where $c_k \in \{ c_n: c_n-c_{n-1}=13(2^{2n}),c_0=17\text{ and }n\in \mathbb{N}\}$ and $(\beta_3-1)/3$ will be the initial term for the general term of the next sequence of numbers. So after two or three simple trials you can get that $c_k=17$ and $\beta_3=17(2^1)=34$ which is the least among others and you can verify that $\beta_3-1=34-1=33$ which is divisible by $3$ and $(\beta_3-1)/3=11=c_0$ is the initial term for the next sequence function. Now let's see the following figure $4$ to be more clear for the next and the previous work. 
\begin{theorem}
Collatz conjecture is true for all odd numbers in the set of the recurrence relation 
\begin{equation*}
 \{ d_n: d_n-d_{n-1}=17(2^{2n-1}),d_0=11\text{ and }n\in \mathbb{N}\}
\end{equation*}
\begin{equation*}
\Rightarrow d_n=\frac{1}{3}\bigg[17(2^{2n+1})-1\bigg],\text{ }n\geq 0
\end{equation*}    and for each $d_{n}$, there are $N_{17}+2n+2$ number of steps (or Collatz function operation) needed to get $1$.   
\end{theorem} 
\begin{proof}
\newpage
\tikzstyle{rect}=[draw,rectangle,fill=blue!20,text width=3em,text centered, minimum height=2em]
\tikzstyle{ellip}=[draw,ellipse,fill=white!20, minimum height=2em]
\tikzstyle{circ}=[draw, circle, fill=white!20, minimum width=2pt,inner sep=3pt]
\tikzstyle{diam}=[draw,diamond,fill=white!20,text width=8em, text badly centered, inner sep=0pt]
\tikzstyle{line}=[draw,-latex']
\begin{figure}[h!]
\begin{center}
\begin{tikzpicture}[node distance=1cm, auto]
\node [rect, rounded corners] (step1) {$17$};
\node [rect, rounded corners, above of=step1,node distance=2cm] (step2) {$17(2)$};
\node [rect, rounded corners, above of=step2,node distance=2cm] (step3) {$17(2^2)$};
\node [rect, rounded corners, above of=step3,node distance=2cm] (step4) {$17(2^3)$};
\node [rect, rounded corners, above of=step4,node distance=2cm] (step5) {$17(2^4)$};
\node [rect, rounded corners, above of=step5,node distance=2cm] (step6) {$17(2^5)$};
\node [rect, rounded corners, above of=step6,node distance=2cm] (step7) {$17(2^6)$};
\node [rect, rounded corners, above of=step7,node distance=2cm] (step8) {$17(2^7)$};
\node [rect, rounded corners, above of=step8,node distance=2cm] (step9) {$17(2^8)$};
\node [rect, rounded corners, above of=step9,node distance=2cm] (step10) {$17(2^9)$};
\node [rect, right of=step2,node distance=2cm,color=red] (step11) {\color{white}$11$};
\node [rect, right of=step11,node distance=2cm] (step12) {$11(2)$};
\node [rect, right of=step12,node distance=2cm] (step13) {$11(2^2)$};
\node [rect, right of=step13,node distance=2cm] (step14) {$11(2^3)$};
\node [rect, right of=step14,node distance=2cm] (step15) {$11(2^4)$};
\node [rect, right of=step15,node distance=2cm] (step16) {$11(2^5)$};
\node [rect, right of=step16,node distance=2cm] (step17) {$11(2^6)$};
\node [rect, right of=step17,node distance=2cm] (step18) {$11(2^7)$};
\node [rect, right of=step4,node distance=2cm,color=red] (step19) {\color{white}$45$};
\node [rect, right of=step19,node distance=2cm] (step20) {$45(2)$};
\node [rect, right of=step20,node distance=2cm] (step21) {$45(2^2)$};
\node [rect, right of=step21,node distance=2cm] (step22) {$45(2^3)$};
\node [rect, right of=step22,node distance=2cm] (step23) {$45(2^4)$};
\node [rect, right of=step23,node distance=2cm] (step24) {$45(2^5)$};
\node [rect, right of=step24,node distance=2cm] (step25) {$45(2^6)$};
\node [rect, right of=step25,node distance=2cm] (step26) {$45(2^7)$};

\node [rect, right of=step6,node distance=2cm,color=red] (step27) {\color{white}$181$};
\node [rect, right of=step27,node distance=2cm] (step28) {$181(2)$};
\node [rect, right of=step28,node distance=2cm] (step29) {$181(2^2)$};
\node [rect, right of=step29,node distance=2cm] (step30) {$181(2^3)$};
\node [rect, right of=step30,node distance=2cm] (step31) {$181(2^4)$};
\node [rect, right of=step31,node distance=2cm] (step32) {$181(2^5)$};
\node [rect, right of=step32,node distance=2cm] (step33) {$181(2^6)$};
\node [rect, right of=step33,node distance=2cm] (step34) {$181(2^7)$};

\node [rect, right of=step8,node distance=2cm,color=red] (step35) {\color{white} $725$};
\node [rect, right of=step35,node distance=2cm] (step36) {\tiny $725(2)$};
\node [rect, right of=step36,node distance=2cm] (step37) {\tiny $725(2^2)$};
\node [rect, right of=step37,node distance=2cm] (step38) {\tiny $725(2^3)$};
\node [rect, right of=step38,node distance=2cm] (step39) {\tiny $725(2^4)$};
\node [rect, right of=step39,node distance=2cm] (step40) {\tiny $725(2^5)$};
\node [rect, right of=step40,node distance=2cm] (step41) {\tiny $725(2^6)$};
\node [rect, right of=step41,node distance=2cm] (step42) {\tiny $725(2^7)$};

\path [line] (step1)--(step2);
\path [line] (step2)--(step3);
\path [line] (step3)--(step4);
\path [line] (step4)--(step5);
\path [line] (step5)--(step6);
\path [line] (step6)--(step7);
\path [line] (step7)--(step8);
\path [line] (step8)--(step9);
\path [line] (step9)--(step10);

\path [line] (step18)--(step17);
\path [line] (step17)--(step16);
\path [line] (step16)--(step15);
\path [line] (step15)--(step14);
\path [line] (step14)--(step13);
\path [line] (step13)--(step12);
\path [line] (step12)--(step11);
\path [line] (step11)--(step2);

\path [line] (step26)--(step25);
\path [line] (step25)--(step24);
\path [line] (step24)--(step23);
\path [line] (step23)--(step22);
\path [line] (step22)--(step21);
\path [line] (step21)--(step20);
\path [line] (step20)--(step19);
\path [line] (step19)--(step4);

\path [line] (step34)--(step33);
\path [line] (step33)--(step32);
\path [line] (step32)--(step31);
\path [line] (step31)--(step30);
\path [line] (step30)--(step29);
\path [line] (step29)--(step28);
\path [line] (step28)--(step27);
\path [line] (step27)--(step6);

\path [line] (step42)--(step41);
\path [line] (step41)--(step40);
\path [line] (step40)--(step39);
\path [line] (step39)--(step38);
\path [line] (step38)--(step37);
\path [line] (step37)--(step36);
\path [line] (step36)--(step35);
\path [line] (step35)--(step8);

\end{tikzpicture}
\caption{\small Proof of Collatz conjecture for the sequence of numbers $\{11,45,181,725\cdots \}$ combined with the previous figures.}
\end{center}
\end{figure}
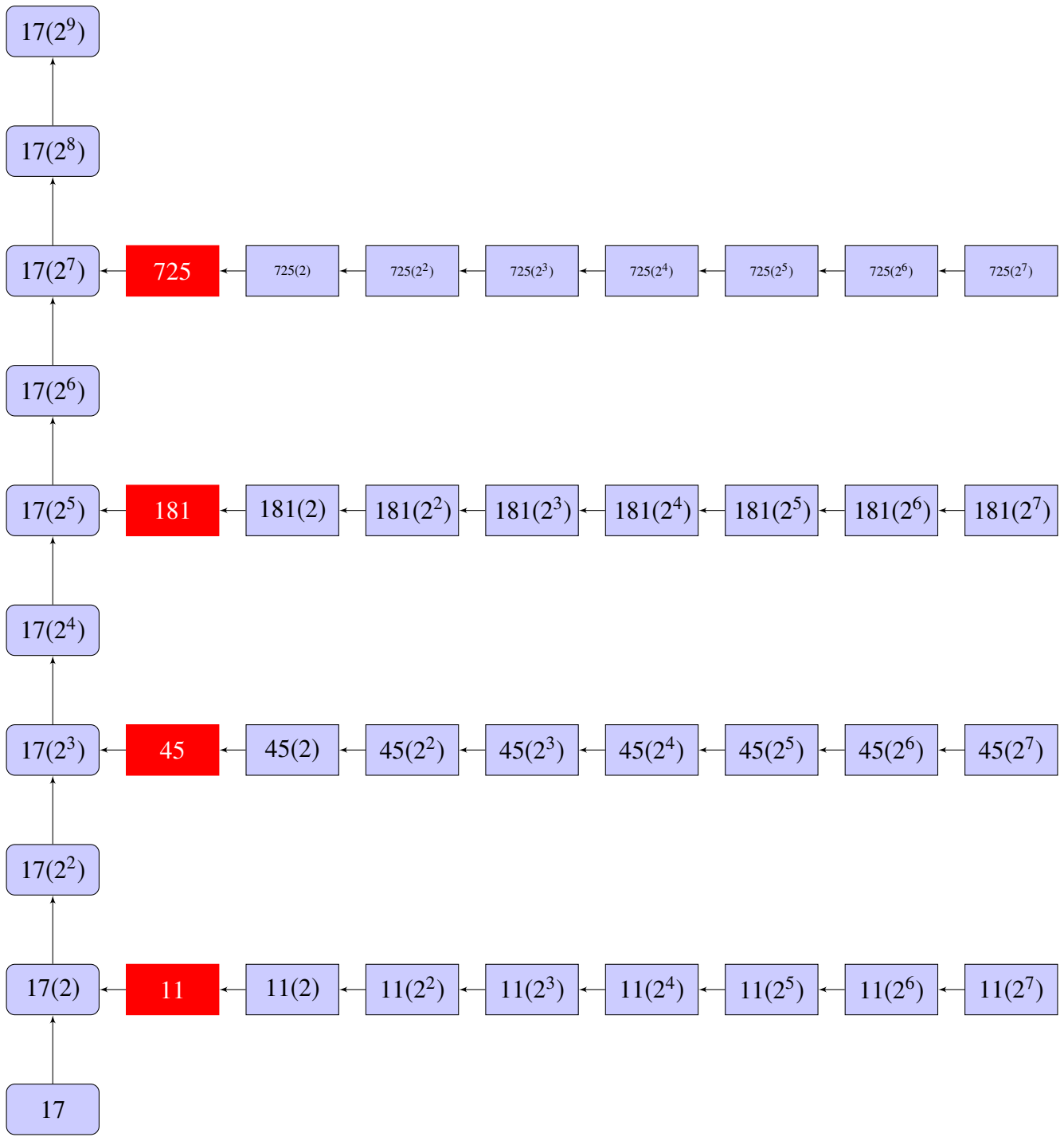

\justify
As you see from the picture we have sequence of numbers $\{11,45,181,725\cdots \}$. All elements of this sequence are odd numbers with $d_0=11,d_1=45,d_2=181,d_3=725\cdots$, and $3(11)+1=17(2^1),3(45)+1=17(2^{3}),3(181)+1=17(2^5),3(725)+1=17(2^7),\cdots$. The question is how this sequence of numbers formed?. The answer is simple. You can observe that $45=11+17(2^1),181=45+17(2^3),725=181+17(2^5),\cdots$. Therefore the recurrence relation is $d_n=d_{n-1}+17(2^{2n-1})$ or $d_n-d_{n-1}=17(2^{2n-1})$ with $d_0=11$ and $n\in \mathbb{N}$.\\ Hence we can conclude that Collatz analysis is true for all numbers in the set
\begin{equation*}
 \{ d_n: d_n-d_{n-1}=17(2^{2n-1}),d_0=11\text{ and }n\in \mathbb{N}\}
\end{equation*}
If we solve this non-homogeneous recurrence relation we get,
\begin{equation*}
d_n=\frac{1}{3}\bigg[17(2^{2n+1})-1\bigg],\text{ }n\geq 0
\end{equation*} 
This is the general term for the above sequence of numbers. We might ask "How many number of steps we need to reach to $1$ for each number in the above sequence by applying Collatz function?". Here is the answer. For example for $n=0$, we have $d_0=11$, therefore according to Collatz conjecture $3(11)+1=34=17(2^1)$ since $d_0$ is odd. As you can see from the above figure that we need to have $14$ steps, which means $2+12$ steps, where $2=1+1$ is the exponent of $2$ plus $1$ from the equation on the right side of $3(d_0)+1=34=17(2^1)$ and $12$ is the number steps needed for $17$, so $d_0=11$ needs to have $14=2+12$ number of steps to reach to $1$ by applying Collatz  function repeatedly. Similarly for $a_1=45$ we need to have $12+4=16$ steps to reach to $1$ since $3(45)+1=17(2^{3})$, that is take the exponent of $2$ on the right hand side of $3(45)+1=17(2^3)$ and add $1$, that is, $3+1$ steps plus $12$ steps. Thus for $d_2=181$ we need to have $18$ steps since $3(181)+1=17(2^{5})$ and $d_3=725$ needs to have $20$ steps since $3(725)+1=17(2^7)$. You can check that this is true for the rest of numbers in the sequence by choosing randomly.
\end{proof}
Next, what I'm going to do is track another sequence function. The question is how to get this sequence function?. Well, the answer is simple if I draw another figure just like the previous but I have to use the previous figures to draw the next. First of all, for $k\in \mathbb{N}$, let me find the least $\beta_4=d_k2^{2n} \text{ or }=d_k2^{2n-1}\text{ for }n=1$ such that $\beta_4-1$ should be divisible by $3$, where $d_k \in \{ d_n: d_n-d_{n-1}=17(2^{2n-1}),d_0=11\text{ and }n\in \mathbb{N}\}$ and $(\beta_4-1)/3$ will be the initial term for the general term of the next sequence of numbers. So after two or three simple trials you can get that $d_k=11$ and $\beta_4=11(2^1)=22$ which is the least among others and you can verify that $\beta_4-1=22-1=21$ which is divisible by $3$ and $(\beta_4-1)/3=7=d_0$ is the initial term for the next sequence function. Now let's see the following figure $5$ to be more clear for the next and the previous work. 
\begin{theorem}
Collatz conjecture is true for all odd numbers in the set of the recurrence relation 
\begin{equation*}
 \{ e_n: e_n-e_{n-1}=11(2^{2n-1}),e_0=7\text{ and }n\in \mathbb{N}\}
\end{equation*}
\begin{equation*}
\Rightarrow e_n=\frac{1}{3}\bigg[11(2^{2n+1})-1\bigg],\text{ }n\geq 0
\end{equation*}    and for each $e_{n}$, there are $N_{11}+2n+2$ number of steps (or Collatz function operation) needed to get $1$.   
\end{theorem} 
\begin{proof}
\newpage
\tikzstyle{rect}=[draw,rectangle,fill=blue!20,text width=3em,text centered, minimum height=2em]
\tikzstyle{ellip}=[draw,ellipse,fill=white!20, minimum height=2em]
\tikzstyle{circ}=[draw, circle, fill=white!20, minimum width=2pt,inner sep=3pt]
\tikzstyle{diam}=[draw,diamond,fill=white!20,text width=8em, text badly centered, inner sep=0pt]
\tikzstyle{line}=[draw,-latex']
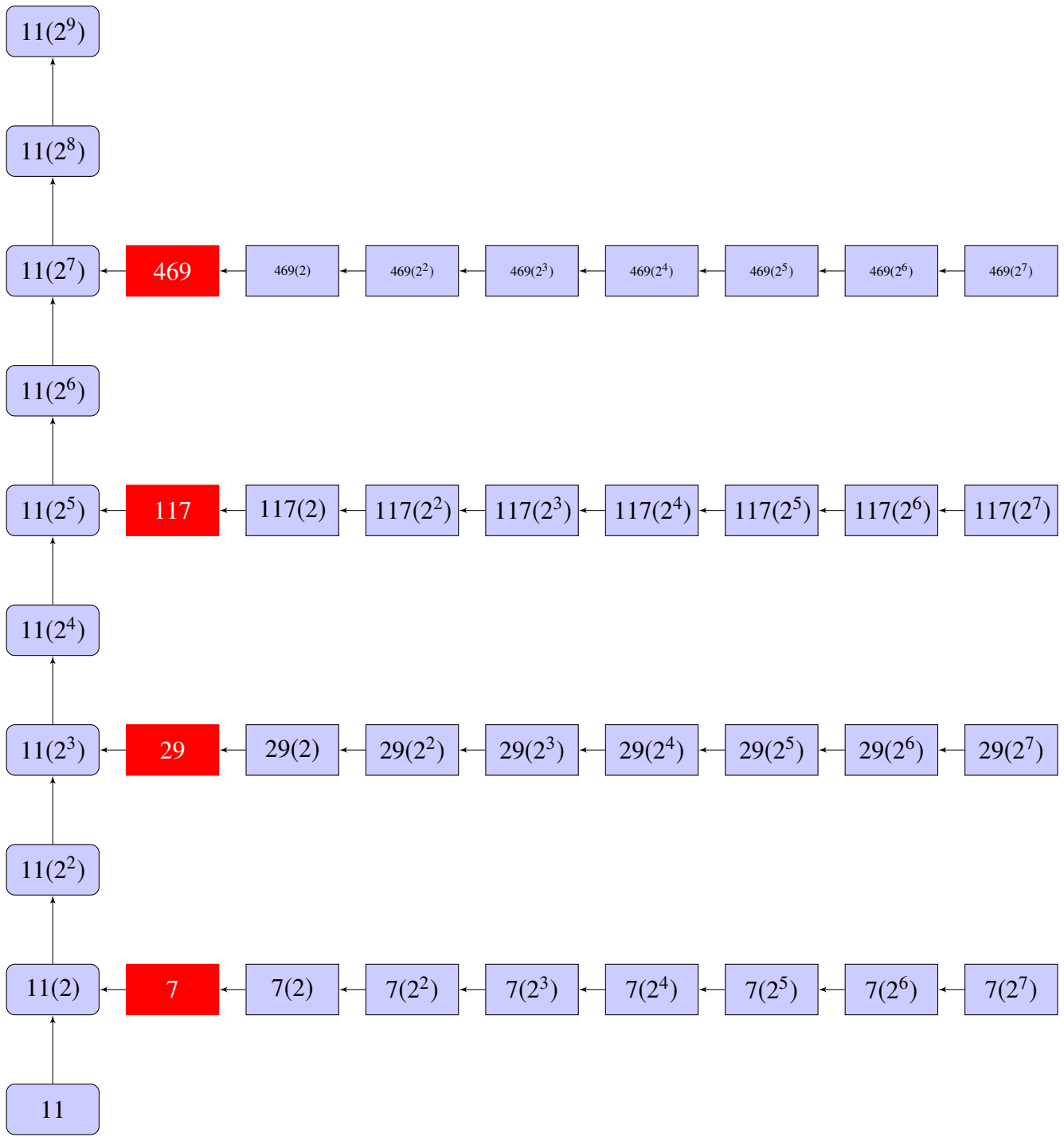
\begin{figure}[h!]
\begin{center}
\begin{tikzpicture}[node distance=1cm, auto]
\node [rect, rounded corners] (step1) {$11$};
\node [rect, rounded corners, above of=step1,node distance=2cm] (step2) {$11(2)$};
\node [rect, rounded corners, above of=step2,node distance=2cm] (step3) {$11(2^2)$};
\node [rect, rounded corners, above of=step3,node distance=2cm] (step4) {$11(2^3)$};
\node [rect, rounded corners, above of=step4,node distance=2cm] (step5) {$11(2^4)$};
\node [rect, rounded corners, above of=step5,node distance=2cm] (step6) {$11(2^5)$};
\node [rect, rounded corners, above of=step6,node distance=2cm] (step7) {$11(2^6)$};
\node [rect, rounded corners, above of=step7,node distance=2cm] (step8) {$11(2^7)$};
\node [rect, rounded corners, above of=step8,node distance=2cm] (step9) {$11(2^8)$};
\node [rect, rounded corners, above of=step9,node distance=2cm] (step10) {$11(2^9)$};
\node [rect, right of=step2,node distance=2cm,color=red] (step11) {\color{white}$7$};
\node [rect, right of=step11,node distance=2cm] (step12) {$7(2)$};
\node [rect, right of=step12,node distance=2cm] (step13) {$7(2^2)$};
\node [rect, right of=step13,node distance=2cm] (step14) {$7(2^3)$};
\node [rect, right of=step14,node distance=2cm] (step15) {$7(2^4)$};
\node [rect, right of=step15,node distance=2cm] (step16) {$7(2^5)$};
\node [rect, right of=step16,node distance=2cm] (step17) {$7(2^6)$};
\node [rect, right of=step17,node distance=2cm] (step18) {$7(2^7)$};
\node [rect, right of=step4,node distance=2cm,color=red] (step19) {\color{white}$29$};
\node [rect, right of=step19,node distance=2cm] (step20) {$29(2)$};
\node [rect, right of=step20,node distance=2cm] (step21) {$29(2^2)$};
\node [rect, right of=step21,node distance=2cm] (step22) {$29(2^3)$};
\node [rect, right of=step22,node distance=2cm] (step23) {$29(2^4)$};
\node [rect, right of=step23,node distance=2cm] (step24) {$29(2^5)$};
\node [rect, right of=step24,node distance=2cm] (step25) {$29(2^6)$};
\node [rect, right of=step25,node distance=2cm] (step26) {$29(2^7)$};

\node [rect, right of=step6,node distance=2cm,color=red] (step27) {\color{white}$117$};
\node [rect, right of=step27,node distance=2cm] (step28) {$117(2)$};
\node [rect, right of=step28,node distance=2cm] (step29) {$117(2^2)$};
\node [rect, right of=step29,node distance=2cm] (step30) {$117(2^3)$};
\node [rect, right of=step30,node distance=2cm] (step31) {$117(2^4)$};
\node [rect, right of=step31,node distance=2cm] (step32) {$117(2^5)$};
\node [rect, right of=step32,node distance=2cm] (step33) {$117(2^6)$};
\node [rect, right of=step33,node distance=2cm] (step34) {$117(2^7)$};

\node [rect, right of=step8,node distance=2cm,color=red] (step35) {\color{white} $469$};
\node [rect, right of=step35,node distance=2cm] (step36) {\tiny $469(2)$};
\node [rect, right of=step36,node distance=2cm] (step37) {\tiny $469(2^2)$};
\node [rect, right of=step37,node distance=2cm] (step38) {\tiny $469(2^3)$};
\node [rect, right of=step38,node distance=2cm] (step39) {\tiny $469(2^4)$};
\node [rect, right of=step39,node distance=2cm] (step40) {\tiny $469(2^5)$};
\node [rect, right of=step40,node distance=2cm] (step41) {\tiny $469(2^6)$};
\node [rect, right of=step41,node distance=2cm] (step42) {\tiny $469(2^7)$};

\path [line] (step1)--(step2);
\path [line] (step2)--(step3);
\path [line] (step3)--(step4);
\path [line] (step4)--(step5);
\path [line] (step5)--(step6);
\path [line] (step6)--(step7);
\path [line] (step7)--(step8);
\path [line] (step8)--(step9);
\path [line] (step9)--(step10);

\path [line] (step18)--(step17);
\path [line] (step17)--(step16);
\path [line] (step16)--(step15);
\path [line] (step15)--(step14);
\path [line] (step14)--(step13);
\path [line] (step13)--(step12);
\path [line] (step12)--(step11);
\path [line] (step11)--(step2);

\path [line] (step26)--(step25);
\path [line] (step25)--(step24);
\path [line] (step24)--(step23);
\path [line] (step23)--(step22);
\path [line] (step22)--(step21);
\path [line] (step21)--(step20);
\path [line] (step20)--(step19);
\path [line] (step19)--(step4);

\path [line] (step34)--(step33);
\path [line] (step33)--(step32);
\path [line] (step32)--(step31);
\path [line] (step31)--(step30);
\path [line] (step30)--(step29);
\path [line] (step29)--(step28);
\path [line] (step28)--(step27);
\path [line] (step27)--(step6);

\path [line] (step42)--(step41);
\path [line] (step41)--(step40);
\path [line] (step40)--(step39);
\path [line] (step39)--(step38);
\path [line] (step38)--(step37);
\path [line] (step37)--(step36);
\path [line] (step36)--(step35);
\path [line] (step35)--(step8);

\end{tikzpicture}
\caption{\small Proof of Collatz conjecture for the sequence of numbers $\{7,29,117,469\cdots \}$ combined with the previous figures.}

\end{center}
\end{figure}

\justify
As you can see from figure $5$, we have sequence of numbers $\{7,29,117,469\cdots \}$. All elements of this sequence are odd numbers with $e_0=7,e_1=29,e_2=117,e_3=469\cdots$, and $3(7)+1=11(2^1),3(29)+1=11(2^{3}),3(117)+1=11(2^5),3(469)+1=11(2^7),\cdots$. The question is how this sequence of numbers formed?. The answer is simple. You can observe that $29=7+11(2^1),117=29+11(2^3),469=117+11(2^5),\cdots$. Therefore the recurrence relation is $e_n=e_{n-1}+11(2^{2n-1})$ or $e_n-e_{n-1}=11(2^{2n-1})$ with $a_0=7$ and $n\in \mathbb{N}$.\\ Hence we can conclude that Collatz analysis is true for all numbers in the set
\begin{equation*}
 \{ e_n: e_n-e_{n-1}=11(2^{2n-1}),e_0=7\text{ and }n\in \mathbb{N}\}
\end{equation*}
If we solve this non-homogeneous recurrence relation we get,
\begin{equation*}
e_n=\frac{1}{3}\bigg[11(2^{2n+1})-1\bigg],\text{ }n\geq 0
\end{equation*} 
This is the general term for the above sequence of numbers. We might ask "How many number of steps we need to reach to $1$ for each number in the above sequence by applying Collatz function?". Here is the answer. For example for $n=0$, we have $e_0=7$, therefore according to Collatz conjecture $3(7)+1=22=11(2^1)$ since $e_0$ is odd. As you can see from the above figures that we need to have $16$ steps, which means $2+14$ steps, where $2=1+1$ is the exponent of $2$ plus $1$ from the equation on the right side of $3(e_0)+1=22=11(2^1)$ and $14$ is the number steps needed for $11$, so $e_0=7$ needs to have $16=2+14$ number of steps to reach to $1$ by applying Collatz  function repeatedly. Similarly for $e_1=29$ we need to have $14+4=18$ steps to reach to $1$ since $3(29)+1=11(2^{3})$, that is take the exponent of $2$ on the right hand side of $3(29)+1=11(2^3)$ and add $1$, that is, $3+1$ steps plus $14$ steps. Thus for $e_2=117$ we need to have $20$ steps since $3(117)+1=11(2^{5})$ and $e_3=469$ needs to have $22$ steps since $3(469)+1=11(2^7)$. You can check that this is true for the rest of numbers in the sequence by choosing randomly.
\end{proof}
Next, what I'm going to do is track another sequence function. The question is how to get this sequence function?. Well, the answer is simple if I draw another figure just like the previous but I have to use the previous figures to draw the next. First of all, for $k\in \mathbb{N}$, let me find the least $\beta_5=e_k2^{2n} \text{ or }=e_k2^{2n-1}\text{ for }n=1$ such that $\beta_5-1$ should be divisible by $3$, where $e_k \in \{ e_n: e_n-e_{n-1}=11(2^{2n-1}),e_0=7\text{ and }n\in \mathbb{N}\}$ and $(\beta_5-1)/3$ will be the initial term for the general term of the next sequence of numbers. So after two or three simple trials you can get that $e_k=7$ and $\beta_5=7(2^2)=28$ which is the least among others and you can verify that $\beta_5-1=28-1=27$ which is divisible by $3$ and $(\beta_5-1)/3=9=e_0$ is the initial term for the next sequence function. Now let's see the following figure $6$ to be more clear for the next and the previous work. 
\begin{theorem}
Collatz conjecture is true for all odd numbers in the set of the recurrence relation 
\begin{equation*}
 \{ f_n: f_n-f_{n-1}=7(2^{2n}),f_0=9\text{ and }n\in \mathbb{N}\}
\end{equation*}
\begin{equation*}
\Rightarrow f_n=\frac{1}{3}\bigg[7(2^{2n+2})-1\bigg],\text{ }n\geq 0
\end{equation*}    and for each $f_{n}$, there are $N_{7}+2n+3$ number of steps (or Collatz function operation) needed to get $1$.   
\end{theorem} 
\begin{proof}
\newpage
\tikzstyle{rect}=[draw,rectangle,fill=blue!20,text width=3em,text centered, minimum height=2em]
\tikzstyle{ellip}=[draw,ellipse,fill=white!20, minimum height=2em]
\tikzstyle{circ}=[draw, circle, fill=white!20, minimum width=2pt,inner sep=3pt]
\tikzstyle{diam}=[draw,diamond,fill=white!20,text width=8em, text badly centered, inner sep=0pt]
\tikzstyle{line}=[draw,-latex']
\begin{figure}[h!]
\begin{center}
\begin{tikzpicture}[node distance=1cm, auto]
\node [rect, rounded corners] (step1) {$7$};
\node [rect, rounded corners, above of=step1,node distance=2cm] (step2) {$7(2)$};
\node [rect, rounded corners, above of=step2,node distance=2cm] (step3) {$7(2^2)$};
\node [rect, rounded corners, above of=step3,node distance=2cm] (step4) {$7(2^3)$};
\node [rect, rounded corners, above of=step4,node distance=2cm] (step5) {$7(2^4)$};
\node [rect, rounded corners, above of=step5,node distance=2cm] (step6) {$7(2^5)$};
\node [rect, rounded corners, above of=step6,node distance=2cm] (step7) {$7(2^6)$};
\node [rect, rounded corners, above of=step7,node distance=2cm] (step8) {$7(2^7)$};
\node [rect, rounded corners, above of=step8,node distance=2cm] (step9) {$7(2^8)$};
\node [rect, rounded corners, above of=step9,node distance=2cm] (step10) {$7(2^9)$};
\node [rect, right of=step3,node distance=2cm,color=red] (step11) {\color{white}$9$};
\node [rect, right of=step11,node distance=2cm] (step12) {$9(2)$};
\node [rect, right of=step12,node distance=2cm] (step13) {$9(2^2)$};
\node [rect, right of=step13,node distance=2cm] (step14) {$9(2^3)$};
\node [rect, right of=step14,node distance=2cm] (step15) {$9(2^4)$};
\node [rect, right of=step15,node distance=2cm] (step16) {$9(2^5)$};
\node [rect, right of=step16,node distance=2cm] (step17) {$9(2^6)$};
\node [rect, right of=step17,node distance=2cm] (step18) {$9(2^7)$};

\node [rect, right of=step5,node distance=2cm,color=red] (step19) {\color{white}$37$};
\node [rect, right of=step19,node distance=2cm] (step20) {$37(2)$};
\node [rect, right of=step20,node distance=2cm] (step21) {$37(2^2)$};
\node [rect, right of=step21,node distance=2cm] (step22) {$37(2^3)$};
\node [rect, right of=step22,node distance=2cm] (step23) {$37(2^4)$};
\node [rect, right of=step23,node distance=2cm] (step24) {$37(2^5)$};
\node [rect, right of=step24,node distance=2cm] (step25) {$37(2^6)$};
\node [rect, right of=step25,node distance=2cm] (step26) {$37(2^7)$};

\node [rect, right of=step7,node distance=2cm,color=red] (step27) {\color{white}$149$};
\node [rect, right of=step27,node distance=2cm] (step28) {$149(2)$};
\node [rect, right of=step28,node distance=2cm] (step29) {$149(2^2)$};
\node [rect, right of=step29,node distance=2cm] (step30) {$149(2^3)$};
\node [rect, right of=step30,node distance=2cm] (step31) {$149(2^4)$};
\node [rect, right of=step31,node distance=2cm] (step32) {$149(2^5)$};
\node [rect, right of=step32,node distance=2cm] (step33) {$149(2^6)$};
\node [rect, right of=step33,node distance=2cm] (step34) {$149(2^7)$};

\node [rect, right of=step9,node distance=2cm,color=red] (step35) {\color{white}$597$};
\node [rect, right of=step35,node distance=2cm] (step36) {\tiny $597(2)$};
\node [rect, right of=step36,node distance=2cm] (step37) {\tiny $597(2^2)$};
\node [rect, right of=step37,node distance=2cm] (step38) {\tiny $597(2^3)$};
\node [rect, right of=step38,node distance=2cm] (step39) {\tiny $597(2^4)$};
\node [rect, right of=step39,node distance=2cm] (step40) {\tiny $597(2^5)$};
\node [rect, right of=step40,node distance=2cm] (step41) {\tiny $597(2^6)$};
\node [rect, right of=step41,node distance=2cm] (step42) {\tiny $597(2^7)$};

\path [line] (step1)--(step2);
\path [line] (step2)--(step3);
\path [line] (step3)--(step4);
\path [line] (step4)--(step5);
\path [line] (step5)--(step6);
\path [line] (step6)--(step7);
\path [line] (step7)--(step8);
\path [line] (step8)--(step9);
\path [line] (step9)--(step10);

\path [line] (step18)--(step17);
\path [line] (step17)--(step16);
\path [line] (step16)--(step15);
\path [line] (step15)--(step14);
\path [line] (step14)--(step13);
\path [line] (step13)--(step12);
\path [line] (step12)--(step11);
\path [line] (step11)--(step3);

\path [line] (step26)--(step25);
\path [line] (step25)--(step24);
\path [line] (step24)--(step23);
\path [line] (step23)--(step22);
\path [line] (step22)--(step21);
\path [line] (step21)--(step20);
\path [line] (step20)--(step19);
\path [line] (step19)--(step5);

\path [line] (step34)--(step33);
\path [line] (step33)--(step32);
\path [line] (step32)--(step31);
\path [line] (step31)--(step30);
\path [line] (step30)--(step29);
\path [line] (step29)--(step28);
\path [line] (step28)--(step27);
\path [line] (step27)--(step7);

\path [line] (step42)--(step41);
\path [line] (step41)--(step40);
\path [line] (step40)--(step39);
\path [line] (step39)--(step38);
\path [line] (step38)--(step37);
\path [line] (step37)--(step36);
\path [line] (step36)--(step35);
\path [line] (step35)--(step9);

\end{tikzpicture}
\caption{\small Proof of Collatz conjecture for the sequence of numbers $\{9,37,149,597\cdots \}$ combined with the previous figures.}

\end{center}
\end{figure}
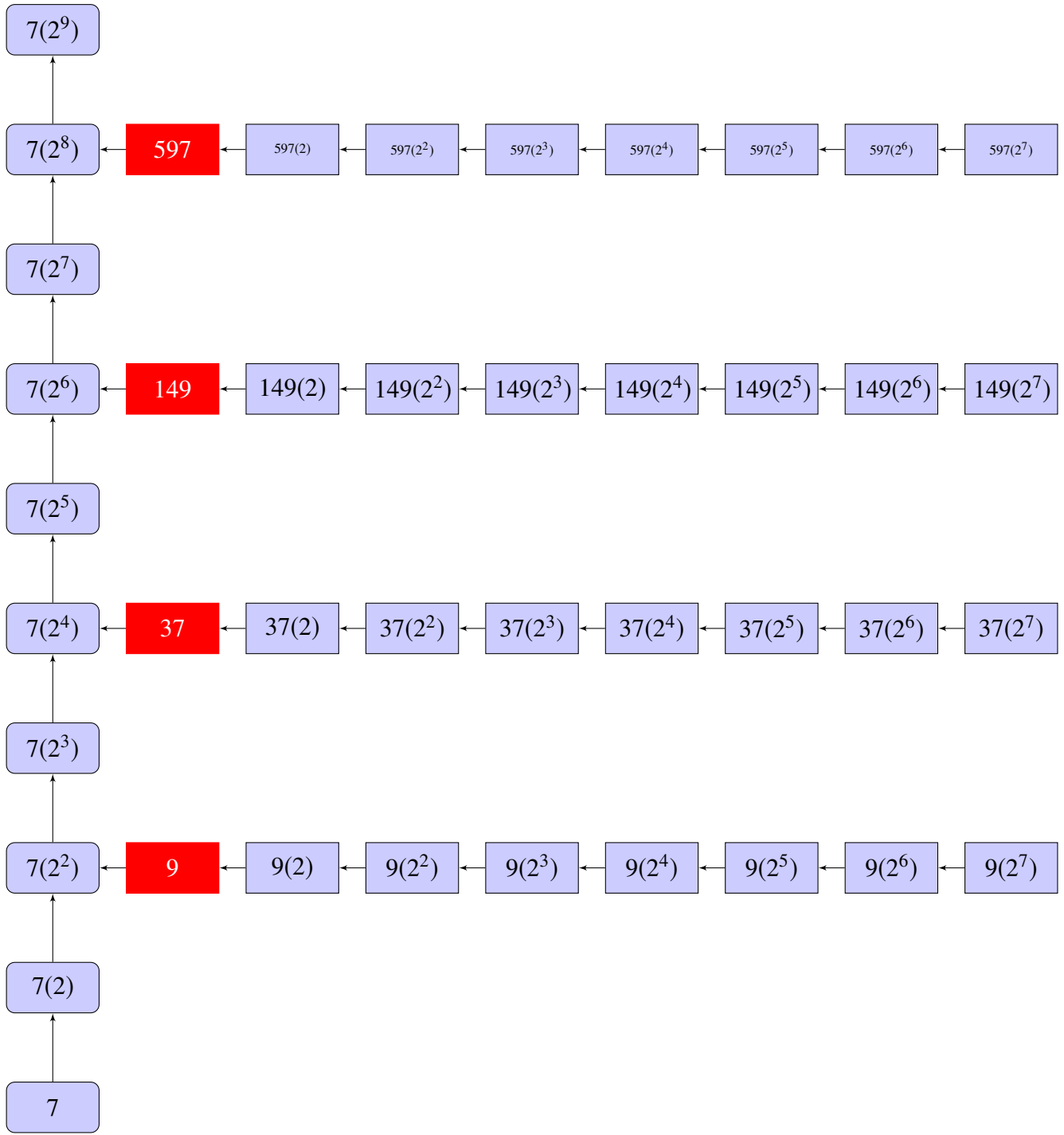

\justify
As you can see from figure $6$, we have sequence of numbers $\{9,37,149,597\cdots \}$. All elements of this sequence are odd numbers with $f_0=9,f_1=37,f_2=149,f_3=597\cdots$, and $3(9)+1=7(2^2),3(37)+1=7(2^{4}),3(149)+1=7(2^6),3(597)+1=11(2^8),\cdots$. The question is how this sequence of numbers formed?. The answer is simple. You can observe that $37=9+7(2^2),149=37+7(2^4),597=149+7(2^6),\cdots$. Therefore the recurrence relation is $f_n=f_{n-1}+7(2^{2n})$ or $f_n-f_{n-1}=7(2^{2n})$ with $f_0=9$ and $n\in \mathbb{N}$.\\ Hence we can conclude that collatz analysis is true for all numbers in the set
\begin{equation*}
 \{ f_n: f_n-f_{n-1}=7(2^{2n}),f_0=9\text{ and }n\in \mathbb{N}\}
\end{equation*}
If we solve this non-homogeneous recurrence relation we get,
\begin{equation*}
f_n=\frac{1}{3}\bigg[7(2^{2n+2})-1\bigg],\text{ }n\geq 0
\end{equation*} 
This is the general term for the above sequence of numbers. We might ask "How many number of steps we need to reach to $1$ for each number in the above sequence by applying Collatz function?". Here is the answer. For example for $n=0$, we have $f_0=9$, therefore according to Collatz conjecture $3(9)+1=28=7(2^2)$ since $f_0$ is odd. As you can see from the above figure that we need to have $19$ steps, which means $3+16$ steps, where $3=2+1$ is the exponent of $2$ plus $1$ from the equation on the right side of $3(f_0)+1=28=7(2^2)$ and $16$ is the number steps needed for $7$, so $f_0=9$ needs to have $19=3+16$ number of steps to reach to $1$ by applying Collatz  function repeatedly. Similarly for $f_1=37$ we need to have $16+5=21$ steps to reach to $1$ since $3(37)+1=7(2^{4})$, that is take the exponent of $2$ on the right hand side of $3(37)+1=7(2^4)$ and add $1$, that is, $4+1$ steps plus $16$ steps. Thus for $f_2=149$ we need to have $23$ steps since $3(149)+1=7(2^{6})$ and $f_3=597$ needs to have $25$ steps since $3(597)+1=7(2^8)$. You can check that this is true for the rest of numbers in the sequence by choosing randomly.
\end{proof}
Next, what I'm going to do is track another sequence function. The question is how to get this sequence function?. Well, the answer is simple if I draw another figure just like the previous but I have to use the previous pictures to draw the next. First of all, for $k\in \mathbb{N}$, let me find the least $\beta_6=f_k2^{2n} \text{ or }=f_k2^{2n-1}\text{ for }n=1$ such that $\beta_6-1$ should be divisible by $3$, where $f_k \in \{ f_n: f_n-f_{n-1}=11(2^{2n-1}),f_0=7\text{ and }n\in \mathbb{N}\}$ and $(\beta_6-1)/3$ will be the initial term for the general term of the next sequence of numbers. So after two or three simple trials you can get that $f_k=29$ and $\beta_6=29(2^1)=58$ which is the least among others and you can verify that $\beta_6-1=58-1=57$ which is divisible by $3$ and $(\beta_6-1)/3=19=f_0$ is the initial term for the next sequence function. Now let's see the following figure $7$ to be more clear for the next and the previous work. 
\begin{theorem}
Collatz conjecture is true for all odd numbers in the set of the recurrence relation 
\begin{equation*}
 \{ g_n: g_n-g_{n-1}=29(2^{2n-1}),g_0=19\text{ and }n\in \mathbb{N}\}
\end{equation*}
\begin{equation*}
g_n=\frac{1}{3}\bigg[29(2^{2n+1})-1\bigg],\text{ }n\geq 0
\end{equation*}    and for each $g_{n}$, there are $N_{29}+2n+2$ number of steps (or Collatz function operation) needed to get $1$.   
\end{theorem}
\begin{proof}
\newpage
\tikzstyle{rect}=[draw,rectangle,fill=blue!20,text width=3em,text centered, minimum height=2em]
\tikzstyle{ellip}=[draw,ellipse,fill=white!20, minimum height=2em]
\tikzstyle{circ}=[draw, circle, fill=white!20, minimum width=2pt,inner sep=3pt]
\tikzstyle{diam}=[draw,diamond,fill=white!20,text width=8em, text badly centered, inner sep=0pt]
\tikzstyle{line}=[draw,-latex']
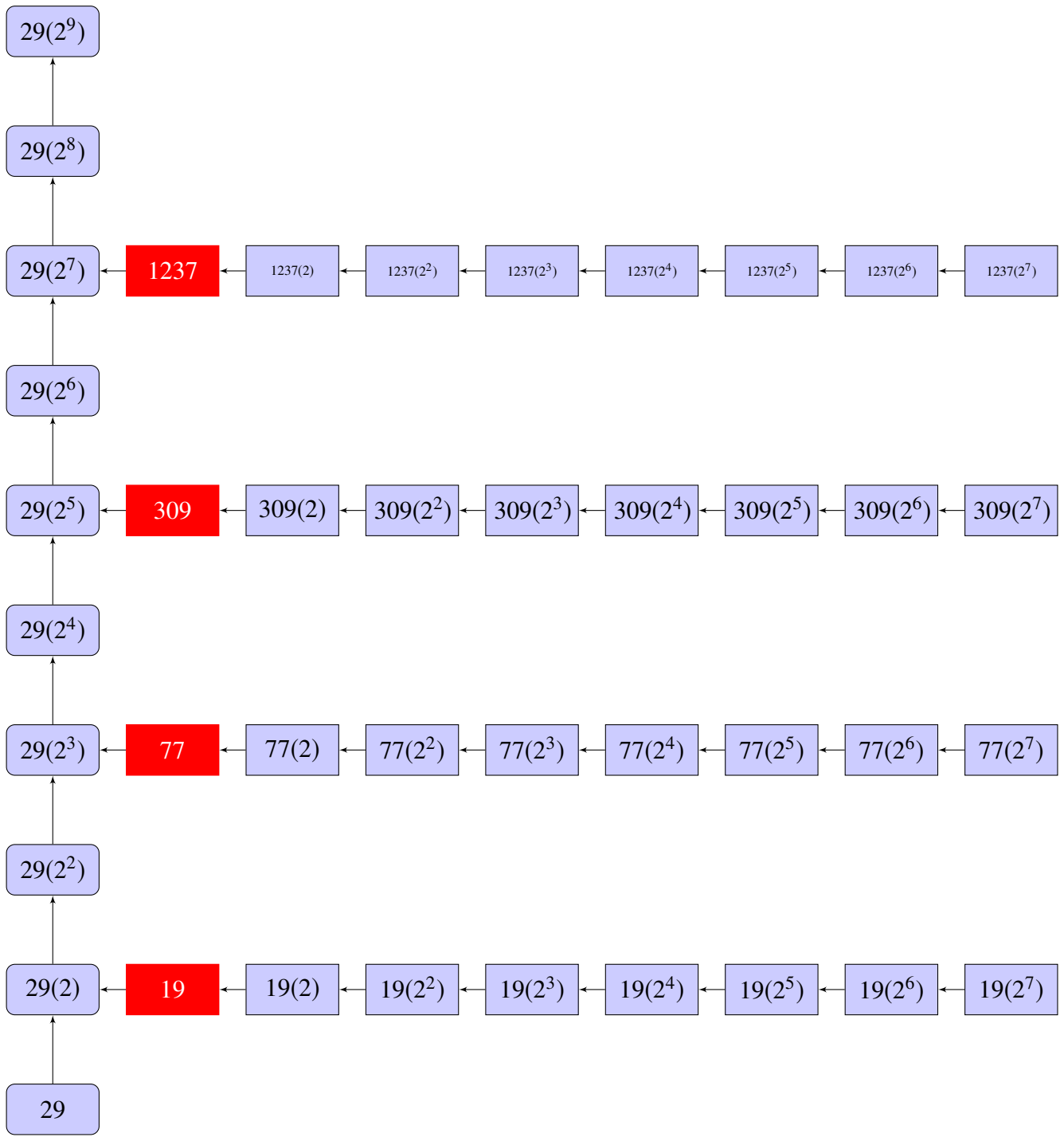
\begin{figure}[h!]
\begin{center}
\begin{tikzpicture}[node distance=1cm, auto]
\node [rect, rounded corners] (step1) {$29$};
\node [rect, rounded corners, above of=step1,node distance=2cm] (step2) {$29(2)$};
\node [rect, rounded corners, above of=step2,node distance=2cm] (step3) {$29(2^2)$};
\node [rect, rounded corners, above of=step3,node distance=2cm] (step4) {$29(2^3)$};
\node [rect, rounded corners, above of=step4,node distance=2cm] (step5) {$29(2^4)$};
\node [rect, rounded corners, above of=step5,node distance=2cm] (step6) {$29(2^5)$};
\node [rect, rounded corners, above of=step6,node distance=2cm] (step7) {$29(2^6)$};
\node [rect, rounded corners, above of=step7,node distance=2cm] (step8) {$29(2^7)$};
\node [rect, rounded corners, above of=step8,node distance=2cm] (step9) {$29(2^8)$};
\node [rect, rounded corners, above of=step9,node distance=2cm] (step10) {$29(2^9)$};
\node [rect, right of=step2,node distance=2cm,color=red] (step11) {\color{white}$19$};
\node [rect, right of=step11,node distance=2cm] (step12) {$19(2)$};
\node [rect, right of=step12,node distance=2cm] (step13) {$19(2^2)$};
\node [rect, right of=step13,node distance=2cm] (step14) {$19(2^3)$};
\node [rect, right of=step14,node distance=2cm] (step15) {$19(2^4)$};
\node [rect, right of=step15,node distance=2cm] (step16) {$19(2^5)$};
\node [rect, right of=step16,node distance=2cm] (step17) {$19(2^6)$};
\node [rect, right of=step17,node distance=2cm] (step18) {$19(2^7)$};
\node [rect, right of=step4,node distance=2cm,color=red] (step19) {\color{white}$77$};
\node [rect, right of=step19,node distance=2cm] (step20) {$77(2)$};
\node [rect, right of=step20,node distance=2cm] (step21) {$77(2^2)$};
\node [rect, right of=step21,node distance=2cm] (step22) {$77(2^3)$};
\node [rect, right of=step22,node distance=2cm] (step23) {$77(2^4)$};
\node [rect, right of=step23,node distance=2cm] (step24) {$77(2^5)$};
\node [rect, right of=step24,node distance=2cm] (step25) {$77(2^6)$};
\node [rect, right of=step25,node distance=2cm] (step26) {$77(2^7)$};

\node [rect, right of=step6,node distance=2cm,color=red] (step27) {\color{white}$309$};
\node [rect, right of=step27,node distance=2cm] (step28) {$309(2)$};
\node [rect, right of=step28,node distance=2cm] (step29) {$309(2^2)$};
\node [rect, right of=step29,node distance=2cm] (step30) {$309(2^3)$};
\node [rect, right of=step30,node distance=2cm] (step31) {$309(2^4)$};
\node [rect, right of=step31,node distance=2cm] (step32) {$309(2^5)$};
\node [rect, right of=step32,node distance=2cm] (step33) {$309(2^6)$};
\node [rect, right of=step33,node distance=2cm] (step34) {$309(2^7)$};

\node [rect, right of=step8,node distance=2cm,color=red] (step35) {\color{white}$1237$};
\node [rect, right of=step35,node distance=2cm] (step36) {\tiny $1237(2)$};
\node [rect, right of=step36,node distance=2cm] (step37) {\tiny $1237(2^2)$};
\node [rect, right of=step37,node distance=2cm] (step38) {\tiny $1237(2^3)$};
\node [rect, right of=step38,node distance=2cm] (step39) {\tiny $1237(2^4)$};
\node [rect, right of=step39,node distance=2cm] (step40) {\tiny $1237(2^5)$};
\node [rect, right of=step40,node distance=2cm] (step41) {\tiny $1237(2^6)$};
\node [rect, right of=step41,node distance=2cm] (step42) {\tiny $1237(2^7)$};

\path [line] (step1)--(step2);
\path [line] (step2)--(step3);
\path [line] (step3)--(step4);
\path [line] (step4)--(step5);
\path [line] (step5)--(step6);
\path [line] (step6)--(step7);
\path [line] (step7)--(step8);
\path [line] (step8)--(step9);
\path [line] (step9)--(step10);

\path [line] (step18)--(step17);
\path [line] (step17)--(step16);
\path [line] (step16)--(step15);
\path [line] (step15)--(step14);
\path [line] (step14)--(step13);
\path [line] (step13)--(step12);
\path [line] (step12)--(step11);
\path [line] (step11)--(step2);

\path [line] (step26)--(step25);
\path [line] (step25)--(step24);
\path [line] (step24)--(step23);
\path [line] (step23)--(step22);
\path [line] (step22)--(step21);
\path [line] (step21)--(step20);
\path [line] (step20)--(step19);
\path [line] (step19)--(step4);

\path [line] (step34)--(step33);
\path [line] (step33)--(step32);
\path [line] (step32)--(step31);
\path [line] (step31)--(step30);
\path [line] (step30)--(step29);
\path [line] (step29)--(step28);
\path [line] (step28)--(step27);
\path [line] (step27)--(step6);

\path [line] (step42)--(step41);
\path [line] (step41)--(step40);
\path [line] (step40)--(step39);
\path [line] (step39)--(step38);
\path [line] (step38)--(step37);
\path [line] (step37)--(step36);
\path [line] (step36)--(step35);
\path [line] (step35)--(step8);

\end{tikzpicture}
\caption{\small Proof of Collatz conjecture for the sequence of numbers $\{19,77,309,1237\cdots \}$ combined with the previous figures.}

\end{center}
\end{figure}

\justify
As you see from the picture we have sequence of numbers $\{19,77,309,1237\cdots \}$. All elements of this sequence are odd numbers with $g_0=19,g_1=77,g_2=309,g_3=1237\cdots$, and $3(19)+1=29(2^1),3(77)+1=29(2^{3}),3(309)+1=29(2^5),3(1237)+1=29(2^7),\cdots$. The question is how this sequence of numbers formed?. The answer is simple. You can observe that $77=19+29(2^1),309=77+29(2^3),1237=309+29(2^5),\cdots$. Therefore the general term is $g_n=g_{n-1}+29(2^{2n-1})$ or $g_n-g_{n-1}=29(2^{2n-1})$ with $g_0=19$ and $n\in \mathbb{N}$.\\ Hence we can conclude that Collatz analysis is true for all numbers in the set
\begin{equation*}
 \{ g_n: g_n-g_{n-1}=29(2^{2n-1}),g_0=19\text{ and }n\in \mathbb{N}\}
\end{equation*}
If we solve this non-homogeneous recurrence relation we get,
\begin{equation*}
g_n=\frac{1}{3}\bigg[29(2^{2n+1})-1\bigg],\text{ }n\geq 0
\end{equation*} 
This is the general term for the above sequence of numbers. We might ask "How many number of steps we need to reach to $1$ for each number in the above sequence by applying Collatz function?". Here is the answer. For example for $n=0$, we have $g_0=19$, therefore according to Collatz conjecture $3(19)+1=58=29(2^1)$ since $g_0$ is odd. As you can see from the above figures that we need to have $20$ steps, which means $2+18$ steps, where $2=1+1$ is the exponent of $2$ plus $1$ from the equation on the right side of $3(g_0)+1=58=29(2^1)$ and $18$ is the number steps needed for $29$, so $g_0=19$ needs to have $20=2+18$ number of steps to reach to $1$ by applying Collatz function repeatedly. Similarly for $g_1=77$ we need to have $18+4=22$ steps to reach to $1$ since $3(77)+1=29(2^{3})$, that is take the exponent of $2$ on the right hand side of $3(77)+1=29(2^3)$ and add $1$, that is, $3+1$ steps plus $18$ steps. Thus for $g_2=309$ we need to have $24$ steps since $3(309)+1=29(2^{5})$ and $g_3=1237$ needs to have $26$ steps since $3(1237)+1=29(2^7)$. You can check that this is true for the rest of numbers in the sequence by choosing randomly.
\end{proof}
\sectionn{Proof of Collatz conjecture for all positive integers }
\label{sec:7}

{ \fontfamily{times}\selectfont
 \noindent}
 \justify
I have proved Collatz conjecture for numbers in the recurrence relations:
\begin{equation*}
a_n-a_{n-1}=2^{2n}\text{ ; }a_0=1
\end{equation*} 

\begin{equation*}
b_n-b_{n-1}=5(2^{2n-1})\text{ ; }b_0=3
\end{equation*}

\begin{equation*}
c_n-c_{n-1}=13(2^{2n-1})\text{ ; }c_0=17
\end{equation*}

\begin{equation*}
d_n-d_{n-1}=17(2^{2n-1})\text{ ; }d_0=11
\end{equation*}

\begin{equation*}
e_n-e_{n-1}=11(2^{2n-1})\text{ ; }e_0=7
\end{equation*}

\begin{equation*}
f_n-f_{n-1}=7(2^{2n})\text{ ; }f_0=9
\end{equation*}

\begin{equation*}
g_n-g_{n-1}=29(2^{2n-1})\text{ ; }g_0=19
\end{equation*}
Hence we can propose the following recurrence relations such that Collatz conjecture is true for all $k\in \mathbb{Z+}$:

\begin{equation}
D_n-D_{n-1}=(3k-1)2^{2n-1}\text{ ; }D_0=2k-1
\end{equation}

\begin{equation*}
\Rightarrow D_n=\frac{1}{3}(2^{2n+1}(3k-1)-1)
\end{equation*}

\begin{equation}
J_n-J_{n-1}=(3k+2)2^{2n-1}\text{ ; }J_0=2k+1
\end{equation}

\begin{equation*}
\Rightarrow J_n=\frac{1}{3}(2^{2n+1}(3k+2)-1)
\end{equation*}

\begin{equation}
M_n-M_{n-1}=(6k-1)2^{2n-1}\text{ ; }M_0=4k-1
\end{equation}

\begin{equation*}
\Rightarrow M_n=\frac{1}{3}(2^{2n+1}(6k-1)-1)
\end{equation*}

\begin{equation}
K_{n}-K_{n-1}=(12k-1)2^{2n-1}\text{ ; }K_0=8k-1
\end{equation}

\begin{equation*}
\Rightarrow K_n=\frac{1}{3}(2^{2n+1}(12k-1)-1)
\end{equation*}

\begin{equation}
S_n-S_{n-1}=(3k+1)2^{2n}\text{ ; }S_0=4k+1
\end{equation}

\begin{equation*}
\Rightarrow S_n=\frac{1}{3}(2^{2n+2}(3k+1)-1)
\end{equation*}

You can easily understand that one recurrence relation is depends on the previous recurrence relation, that is to formulate the next recurrence relation formulae we have to use numbers in the previous recurrence relation formulas. Actually we have discussed this in detail earlier.\\  
Here we have to ask a big question "Is it possible to generalize the above recurrence relations for all $\alpha_0=2k-1=n$ such that Collatz analysis is true for all $k \in \mathbb{N}$, that is, the existence of $\alpha_0=n=2k-1$ for all $k \in \mathbb{N}$ where Collatz conjecture is true". If we can show the existence of $n=2k-1$ for all $k \in \mathbb{N}$ and if we generalize the above recurrence relations that we have discussed above by one recurrence relation formula for all odd natural numbers $\alpha_0=n$, then Collatz conjecture becomes Collatz Theorem. Let's get started from the proof of the Theorem that shows the existence of $n=2k-1=(\beta-1)/3, \text{ }k\in \mathbb{N}$, where $\beta=2\alpha_k$ or $\beta=4\alpha_k$ such that $\beta-1$ is multiple of $3$ or divisible by $3$ as we have discussed before.
\begin{theorem}
If we can find the sequence $\alpha_k$ or $2\alpha_k$ for all $k \in \mathbb{N}$ where $\beta=2\alpha_k$ or $\beta=4\alpha_k$ such that $\beta-1$ is a multiple of $3$ or divisible by $3$, then $n=2k-1$ existed for all $k \in \mathbb{N}$. Where 
$\alpha_k \in \{ \alpha_n:\alpha_n-\alpha_{n-1}=(3s-1)2^{2n-1},\alpha_0=2s-1\text{ and }n,s\in \mathbb{N}\text{ and }s \text{ is fixed}\}$ and $(\beta-1)/3$ will be the initial term for the general term of the next sequence of numbers, that is, if $(\beta-1)/3=2k-1$ and $\alpha_k=3k-1$ then we can form
\begin{equation*}
D_n-D_{n-1}=(3k-1)2^{2n-1}\text{ ; }D_0=2k-1
\end{equation*}
\begin{equation*}
\Rightarrow D_n=\frac{1}{3}(2^{2n+1}(3k-1)-1)
\end{equation*}
\end{theorem}

\begin{proof}
\begin{itemize}
\item {\bf Case I:} When $\beta=2\alpha_k$,\\
Suppose $n=2k-1=(\beta-1)/3 \text { for all }k\in \mathbb{N}$
\begin{equation*}
\Longrightarrow 2k-1=\frac{2\alpha_k-1}{3}
\end{equation*}
\begin{equation*}
\Longrightarrow 2k=\frac{2\alpha_k-1}{3}+1
\end{equation*}

\begin{equation*}
\Longrightarrow \alpha_k=3k-1
\end{equation*}
Therfore $\beta=2\alpha_k=2(3k-1)=6k-2 \Longrightarrow \alpha_0=(\beta-1)/3=2k-1$ 
\item {\bf Case II:} When $\beta=4\alpha_k$,\\
Suppose $n=2k-1=(\beta-1)/3 \text { for all }k\in \mathbb{N}$
\begin{equation*}
\Longrightarrow 2k-1=\frac{4\alpha_k-1}{3}
\end{equation*}
\begin{equation*}
\Longrightarrow 2k=\frac{4\alpha_k-1}{3}+1
\end{equation*}

\begin{equation*}
\Longrightarrow 2\alpha_k=3k-1
\end{equation*}
Therfore $\beta=4\alpha_k=2(3k-1)=6k-2 \Longrightarrow \alpha_0=(\beta-1)/3=2k-1$\\[2mm]
We can also suppose $n=4k+1=(\beta-1)/3 \text { for all }k\in \mathbb{N}$
\begin{equation*}
\Longrightarrow 4k+1=\frac{4\alpha_k-1}{3}
\end{equation*}
\begin{equation*}
\Longrightarrow 4k=\frac{4\alpha_k-1}{3}-1
\end{equation*}

\begin{equation*}
\Longrightarrow \alpha_k=3k+1
\end{equation*}
Therfore $\beta=4\alpha_k=4(3k+1)=12k+4 \Longrightarrow \alpha_0=(\beta-1)/3=4k+1$
\end{itemize}

\end{proof}

\begin{theorem}
For all $n=2k-1\text{, }k \in \mathbb{N}$, Using collatz analysis function, we can form the sequence function that generalizes the above recurrence relations, that is,
\begin{equation*}
a_m=\frac{1}{3}\bigg[ 4^m(3n+1)-1\bigg]\text{, }m\geq 0
\end{equation*}
\end{theorem}

\begin{proof}
As we have discussed above, it is always true that for $a_0=n$, we have the following recurrence relations
\begin{equation*}
a_k-a_{k-1}=3a_{k-1}+1\text{, }a_0=n
\end{equation*}
\begin{equation*}
\Longrightarrow a_k=4a_{k-1}+1\text{, }a_0=n
\end{equation*}

which means 
\begin{equation*}
\text{when }k=1\text{, }a_1=4a_0+1=4n+1=4^1n+1
\end{equation*}

\begin{equation*}
\text{when }k=2\text{, }a_2=4(a_1)+1=4(4n+1)+1=16n+5=4^2n+5
\end{equation*}

\begin{equation*}
\text{when }k=3\text{, }a_3=4(a_2)+1=4(16n+5)+1=64n+21=4^3n+21
\end{equation*}

\begin{equation*}
\text{when }k=4\text{, }a_4=4(a_3)+1=4(64n+21)+1=256n+85=4^4n+85
\end{equation*}

\begin{equation*}
\text{when }k=5\text{, }a_5=4(a_4)+1=4(256n+85)+1=1024n+341=4^5n+341
\end{equation*}
\begin{equation*}
\cdots
\end{equation*}
Therefore we have sequence of numbers for all $n=2k-1\text{, }k\in \mathbb{N}$,
\begin{equation*}
n\text{, }4^1n+1\text{, }4^2n+5\text{, }4^3n+21\text{, }4^4n+85\text{, }4^5n+341\text{, }\cdots
\end{equation*}
But we know that the sequence of numbers $0,1,5,21,85,341, \cdots$ is equal to 
\begin{equation*}
\frac{1}{3}\bigg[ 4^m-1\bigg]\text{, }m\geq 0
\end{equation*}
Hence we obtain the general term of the sequence,
\begin{equation*}
n\text{, }4^1n+1\text{, }4^2n+5\text{, }4^3n+21\text{, }4^4n+85\text{, }4^5n+341\text{, }\cdots
\end{equation*} becomes 
\begin{equation*}
a_m=4^mn+\frac{1}{3}\bigg[ 4^m-1\bigg]\text{, }m\geq 0
\end{equation*} 

\begin{equation*}
\Longrightarrow a_m= \frac{1}{3}\bigg[ (3n+1)4^m-1\bigg]\text{, }m\geq 0
\end{equation*}
\end{proof}
Now we have to discuss on the generalization of the number of Collatz function operation needed for each given odd natural number $n$. Let's get started from the $\bf{Lemma\text{ }1}$.
\begin{lemma}
The number of Collatz function operation needed for $1$ is $3$. That is,
\begin{equation*}
N_{1}=3
\end{equation*}
\end{lemma}
\begin{proof}
Since $1$ is odd, then we have to multiply it by $3$ and add $1$, that is, $3(1)+1$ gives us an even number $4$, then divide it by $2$, that is, $4/2$ gives us an even number $2$ again, then divide it by $2$, that is, $2/2$ gives us $1$. Therefore we have used Collatz function $3$ times. Hence $N_{1}=3$.  
\end{proof}

\begin{theorem}
For each $k=1,2,3,\cdots$, the number of Collatz function operation needed for $2k-1$ is $2+N_{(3k-1)}$. That is, we can state the recurrence relation:
\begin{equation*}
N_{(2k-1)}=2+N_{(3k-1)}\text{ ; }N_{1}=3
\end{equation*}
\end{theorem}
\begin{proof}
We can proof this Theorem using mathematical induction. For $k=1$, we have
\begin{equation*}
N_{(2(1)-1)}=2+N_{(3(1)-1)}
\end{equation*}
\begin{equation*}
\Rightarrow N_{1}=2+N_{2}
\end{equation*}
\begin{equation*}
\Rightarrow 3=2+N_{2}
\end{equation*}
\begin{equation*}
\Rightarrow N_{2}=3-2=1
\end{equation*}
Therefore for $k=1$, the recurrence relation is true.
\justify Now let's prove for $k=n$, that is,
\justify 
Since $2n-1$ is odd, then we have to multiply it by $3$ and add $1$, that is, $3(2n-1)+1$ gives us an even number $2(3n-1)$, then divide it by $2$, that is, $2(3n-1)/2$ gives us a number $3n-1$. Therefore we have used Collatz function $2$ times. Hence $N_{(2n-1)}=2+N{(3n-1)}$. Therefore for $k=n$, the recurrence relation is true.
\justify
 Now let's prove for $k=n+1$, that is,
 \justify 
Since $2(n+1)-1=2n+1$ is odd, then we have to multiply it by $3$ and add $1$, that is, $3(2n+1)+1$ gives us an even number $2(3n+2)$, then divide it by $2$, that is, $2(3n+2)/2$ gives us a number $3n+2$. Therefore we have used Collatz function $2$ times. Hence $N_{(2n+1)}=2+N{(3n+2)}$. Therefore for $k=n+1$, the recurrence relation is true, that is,
\begin{equation*}
N_{(2(n+1)-1)}=2+N_{(3(n+1)-1)}
\end{equation*}
\begin{equation*}
\Rightarrow N_{(2n+1)}=2+N_{(3n+2)}\text{ ; }N_3=7
\end{equation*}
\end{proof}

\begin{theorem}
\label{7.1}
For each $k=1,2,3,\cdots$, the number of Collatz function operation needed for $D_n$ is $2n+N_{D_{0}}$. That is, 
\begin{equation*}
N_{D_{n}}=2n+N_{D_{0}}
\end{equation*}where $D_{0}=2k-1$ and $D_n$ is as defined as in equation {\bf \color{blue}(1)}, that is, for $k=1,2,3,\cdots$
\begin{equation*}
D_n=\frac{1}{3}(2^{2n+1}(3k-1)-1)
\end{equation*}
\end{theorem}
\begin{proof}
We can proof this Theorem using mathematical induction. For $n=1$, we have
\begin{equation*}
N_{D_1}=2+N_{D_0}
\end{equation*}
\begin{equation*}
\Rightarrow N_{\left( \frac{1}{3}(2^{3}(3k-1)-1)\right)}=2+N_{2k-1}
\end{equation*}
Since the number $D_1$ is odd, then we have to multiply by $3$ and add $1$, that is, $3D_1+1$ gives us an even number $2^3(3k-1)$, then divide it by $2$ (three times), that is, $2^3(3k-1)/2^3$ gives us a number $3k-1$. Therefore we applied Collatz function $4$ times. Hence $N_{D_1}=4+N_{3k-1}$. Therefore 
\begin{equation*}
N_{D_1}=2+N_{D_0}
\end{equation*}
\begin{equation*}
\Rightarrow 4+N_{(3k-1)}=2+N_{D_0}=2+N_{(2k-1)}
\end{equation*}
\begin{equation*}
\Rightarrow 2+N_{(3k-1)}=N_{(2k-1)}
\end{equation*}

Therefore for $n=1$, the recurrence relation is true.
\justify Now let's prove for $n=m$, that is,
\justify 
Since $D_{m}$ is odd, then we have to multiply it by $3$ and add $1$, that is, $3D_{m}+1$ gives us an even number $2^{2m+1}(3k-1)$, then divide it by two ($2m+1$ times), that is, $2^{2m+1}(3k-1)/2^{2m+1}$ gives us a number $3k-1$. Therefore we have operated Collatz function $2m+2$ times. Hence $N_{D_m}=2m+2+N{(3k-1)}=2m+N_{(2k-1)}=2m+N_{D_0}$. Therefore for $n=m$, the relation is true.
\justify
 Now let's prove for $n=m+1$, that is,
 \justify 
Since $D_{m+1}$ is odd, then we have to multiply it by $3$ and add $1$, that is, $3D_{m+1}+1$ gives us an even number $2^{2(m+1)+1}(3k-1)$, then divide it by two ($2(m+1)+1$ times), that is, $2^{2(m+1)+1}(3k-1)/2^{2(m+1)+1}$ gives us a number $3k-1$. Therefore we have operated Collatz function $2(m+1)+2$ times. Hence $N_{D_{m+1}}=2(m+1)+2+N{(3k-1)}=2m+2+2+N_{(3k-1)}=2m+2+N_{(2k-1)}=2(m+1)+N_{D_0}$. Therefore for $n=m+1$, the relation is true, that is,
\begin{equation*}
N_{D_{n}}=2n+N_{D_0}
\end{equation*}
\begin{equation*}
\Rightarrow N_{D_{m+1}}=2(m+1)+N_{D_0}
\end{equation*}
\end{proof}

\begin{theorem}
If $3n+1=2^{s_1}b_1$, $3(b_1)+1=2^{s_2}b_2$, $3(b_2)+1=2^{s_3}b_3$, $\cdots$, $3(b_{k-1})+1=2^{s_k}$, then for the given odd natural number $n$, it takes $s_1+s_2+\cdots +s_k +k$ steps to reach to $1$ by applying Collatz function
\begin{equation*}
f(n)=\left\{\begin{array}{ll}
n/2 & \text{ if }n \text{ is even,}\\
3n+1& \text{ if }n \text{ is odd.}
\end{array}\right.
\end{equation*}
repeatedly. 
\end{theorem}

\begin{proof}
Since $3n+1=2^{s_1}b_1$ is even, then divide $2^{s_1}b_1$ by $2$, $s_1$ times to get $b_1$. So for the odd number $n$ we used $s_1+1$ steps to get an odd number $b_1$ or we operate Collatz function $s_1+1$ times to get $b_1$. Now we have to multiply $b_1$ by $3$ and add $1$ since $b_1$ is an odd number, that is, $3(b_1)+1=2^{s_2}b_2$. Now since $2^{s_2}b_2$ is an even number, therefore we have to operate by Collatz function $s_2$ times to get $b_2$ or divide $2^{s_2}b_2$ by $2$, $s_2$ times. Thus we used $s_2+1$ steps for the odd number $b_1$ to get $b_2$. Similarly we have to operate Collatz function $s_3+1$ times to get $b_3$ starting from $3(b_2)+1$, operate Collatz function $s_4+1$ times to get $b_4$ starting from $3(b_3)+1$, operate Collatz function $s_5+1$ times to get $b_5$ starting from $3(b_4)+1$, $\cdots$, operate Collatz function $s_k+1$ times to get $b_k=1$ starting from $3(b_{k-1})+1$. Hence we used $(s_1+1)+(s_2+1)+(s_3+1)+(s_4+1)+\cdots+(s_k+1)=s_1+s_2+s_3+\cdots+s_k +k$ steps to get $1$ for the given number $n$, or we operated Collatz function $s_1+s_2+s_3+\cdots+s_k +k$ times to get $1$ starting from $3n+1$. Hence
\begin{equation*}
N_n=s_1+s_2+s_3+\cdots+s_k +k
\end{equation*}
\end{proof}    
\sectionn{Examples}
\label{Se:8}
\begin{example}
Find $N_{8k-3}$.  
\end{example}
{\bf Solution.}
Since $N_{D_{n}}=2n+N_{D_{0}}$ this implies $N_{D_{1}}=2+N_{D_{0}}$, where 
\begin{equation*}
D_n=\frac{1}{3}(2^{2n+1}(3k-1)-1)
\end{equation*}
\begin{equation*}
\Rightarrow D_{1}=\frac{1}{3}(2^{3}(3k-1)-1)=2^{3}k-\frac{1}{3}(2^{3}+1)=2^{3}k-3
\end{equation*}Therefore $N_{8k-3}=2+N_{2k-1}$.
\begin{example}
Find $N_{109}$ and $N_{437}$.
\end{example}
{\bf Solution.} From $N_{8k-3}=2+N_{2k-1}$, take $k=14$ and we know that $N_{27}=111$. Therefore $N_{109}=2+N_{27}=113$. Similarly, $N_{437}=2+113=115$.

\sectionn{An alternative proof of Collatz conjecture}
\label{Se:9}
We redefined the Collatz conjecture in the following theorem~(\ref{9.5}) will be used in the algorithm for determining the numbers in the same line or same family of numbers which satisfies the Collatz conjecture. For example, 5,16,8,4,2,1 are in the same family which are obtained by applying Collatz function repeatedly with initial term, 5. Naturally, Collatz conjecture is stated by iteration of the form $y_{(n+1)}=f(y_{(n)})$, $n=1,2,\cdots$, where $f$ is defined in the following theorem~(\ref{9.5}) and $y_{(n+1)}$ is always converges to $1$ as $n\longrightarrow \infty$ for any natural number $y_{0}$ as an initial term. Theorem (\ref{9.5}) will be used to proof Collatz conjecture in theorem (\ref{9.12}).
\begin{theorem}
\label{9.5}
Collatz conjecture can be redefined as
$y_{(n+1)}=f(y_{(n)})$, $n=1,2,\cdots$, $y_{(n+1)}$ is always converges to $1$ as $n\longrightarrow \infty$ for any natural number $y_{0}$ as an initial number. Where, 
\begin{equation}
\label{9.4}
f(n)=\frac{1}{4}\left( 7n+2+(5n+2)(-1)^{n+1}\right)
\end{equation}
\end{theorem}
\begin{proof}
The function, $f(n)$ of Collatz conjecture can always be written as
\begin{equation}
\label{9.2}
f(n)=3n\sin^2\left(\frac{n\pi}{2}\right)+\sin^2\left(\frac{n\pi}{2}\right)+\frac{n}{2}\cos^2\left(\frac{n\pi}{2}\right),
\end{equation}
but using the trigonometric identity,
\begin{equation*}
\sin^2\left(\frac{n\pi}{2}\right)=\frac{1}{2}(1-\cos(n\pi)),\text{ } \cos^2\left(\frac{n\pi}{2}\right)=\frac{1}{2}(1+\cos(n\pi)),
\end{equation*}
then equation~(\ref{9.2}) becomes,
\begin{equation*}
f(n)=\frac{3n}{2}(1-\cos(n\pi))+\frac{1}{2}(1-\cos(n\pi))+\frac{n}{4}(1+\cos(n\pi)),
\end{equation*}
this implies
\begin{equation}
\label{9.3}
f(n)=\frac{1}{4}\left( 7n+2-(5n+2)\cos(n\pi)\right),
\end{equation}
this means

\begin{equation*}
f(n)=\frac{1}{4}\left( 7n+2+(5n+2)(-1)^{n+1}\right), 
\end{equation*}
and we can directly deduce from the statement of the Collatz conjecture that
\begin{equation}
 y_{(n+1)}=f(y_{(n)})\text{, }y_0=m\in\mathbb{Z}^{+},\text{ } n=1,2,\cdots,
\end{equation}
this implies
\begin{equation}
\begin{split}
\label{9.6} 
 y_{(n+1)}=\frac{1}{4}\left( 7y_{(n)}+2+(5y_{(n)}+2)(-1)^{y_{(n)}+1}\right)\text{, }y_0=m\in\mathbb{Z}^{+},\text{ } n=1,2,\cdots,
\end{split}
\end{equation} 
Therefore Collatz conjecture stated as: $y_{(n+1)}$ is always converges to $1$ as $n\longrightarrow \infty$ for any given positive integer $y_{0}=m$.
\end{proof}
In the following theorem (\ref{9.12}) we prove Collatz conjecture using mathematical induction.
\begin{theorem}Collatz conjecture stated as: If we define a function $f(n)$ by
\label{9.12}
\begin{equation*}
f(n)=\left\{\begin{array}{ll}
n/2 & \text{ if }n \text{ is even,}\\
3n+1& \text{ if }n \text{ is odd.}
\end{array}\right.
\end{equation*}
Form a sequence by performing this operation repeatedly, beginning with
any positive integer, $y_0$, then $f(n)$ always converges to 1.
\end{theorem}
\begin{proof}
Let's proceed using mathematical induction. From theorem (\ref{9.5}), Collatz conjecture redefined as
$y_{(n+1)}=f(y_{(n)})$, $n=1,2,\cdots$, $y_{(n+1)}$ is always converges to $1$ as $n\longrightarrow \infty$ for any natural number $y_{0}$ as an initial number. Where, 
\begin{equation*}
f(n)=\frac{1}{4}\left( 7n+2+(5n+2)(-1)^{n+1}\right)
\end{equation*} 
Now, let's suppose that Collatz conjecture is true for some natural number $n=k$, that is, 
\begin{equation}
\label{9.13}
\begin{split} 
 y_{(k+1)}=\frac{1}{4}\left( 7y_{(k)}+2+(5y_{(k)}+2)(-1)^{y_{(k)}+1}\right)\text{, }y_0=m\in\mathbb{Z}^{+},\text{ } k=1,2,\cdots,
\end{split}
\end{equation} 
  $y_{(k+1)}$ is always converges to $1$ as $k\longrightarrow \infty$ for any given positive integer $y_{0}=m$. Now, let us proof this for $n=k+1$ to complete the proof,
\begin{equation*}
\begin{split} 
 y_{(k+2)}=\frac{1}{4}\left( 7y_{(k+1)}+2+(5y_{(k+1)}+2)(-1)^{y_{(k+1)}+1}\right)\text{, }y_0=m\in\mathbb{Z}^{+},\text{ } k=1,2,\cdots,
\end{split}
\end{equation*}but from equation (\ref{9.13}), $y_{(k+1)}$ converges to 1 as $k\longrightarrow \infty$ for any natural number $y_0=m$. Therefore, 
\begin{equation*}
\begin{split} 
 y_{(k+2)}\longrightarrow \frac{1}{4}\left( 7+2+(5+2)(-1)^{1+1}\right)=4\longrightarrow 1.
\end{split}
\end{equation*}
This completes the proof.
\end{proof}
\sectionn{Algorithms}
\label{Se:10}
We can use the following GNU octave code of computer algorithm to execute the same family of numbers which satisfies Collatz conjecture using the iteration defined in  theorem~(\ref{9.5}). Of course, the algorithm can also be used to check the truthfulness of the conjecture for particular integer, y0. 
\begin{algorithm}{(Collatz Conjecture using octave code)}
\label{9.10}
\justify
function []=collatzfunction(y0,n)\\
y0=input(`Please enter your favourite positive integer? ');\\
n=input(`How many number of iterations would you like to have? ');\\
y(1)=y0;\\
for i=2:n\\
    y(i)=(1/4)*(7*y(i-1)+2-(5*y(i-1)+2)*cos(y(i-1)*pi));\\
    endfor\\
    y
\end{algorithm}
The following GNU octave code is also important to determine the number of steps needed to reach to 1 for the particular integers, $Dn$. Theorem~(\ref{7.1}) is used to construct the code. Then go back to Algorithm 1 to check the validity of number of steps needed to reach to 1 for the integers, $D_n$ and the sequences as well.\\
\begin{algorithm}{(To determine number of steps to reach to 1 of integer D(i))}
\label{9.11}
\justify
function [] =numberofsteps(k,ND0,n)\\
  k=input(`k:= ');\\
  ND0=input(`ND0:= ');\\
  n=input(`n:= ');\\
 D0=2*k-1;\\
 D(1)=8*k-3;\\
 N(D(1))=2+ND0;\\
 m=3*k-1;\\
 for i=2:n\\
D(i)=D(i-1)+m*($2^{(2*i-1)}$);\\
N(D(i))=2*i+ND0;\\
  end\\
$[D;N(D)]$
\end{algorithm}






\newpage
\bibliographystyle{model1-num-names}

\end{document}